\PassOptionsToPackage{dvipsnames}{xcolor}
\documentclass[12pt]{article}

\usepackage{amsmath,amsthm,amsfonts}
\usepackage{mathtools}
\usepackage{latexsym}
\usepackage{tikz,tikz-cd}
\usepackage{comment}
\usepackage{tikzit}				

\tikzstyle{morphism}=[fill=white, draw=black, shape=rectangle]
\tikzstyle{medium box}=[fill=white, draw=black, shape=rectangle, minimum width=0.7cm, minimum height=0.7cm]
\tikzstyle{large morphism}=[fill=white, draw=black, shape=rectangle, minimum width=1.7cm, minimum height=1cm]
\tikzstyle{bn}=[fill=black, draw=black, shape=circle, inner sep=1.5pt]
\tikzstyle{state}=[fill=white, draw=black, regular polygon, regular polygon sides=3, minimum width=0.8cm, shape border rotate=180, inner sep=0pt]
\tikzstyle{medium state}=[fill=white, draw=black, regular polygon, regular polygon sides=3, minimum width=1.3cm, inner sep=0pt, shape border rotate=180]
\tikzstyle{large state}=[fill=white, draw=black, regular polygon, regular polygon sides=3, minimum width=2.2cm, shape border rotate=180, inner sep=0pt]
\tikzstyle{wide state}=[fill=white, draw=black, shape=isosceles triangle, minimum width=0.8cm, shape border rotate=270, inner sep=1.4pt, minimum height=0.5cm, isosceles triangle apex angle=80]
\tikzstyle{wn}=[fill=white, draw=black, shape=circle, inner sep=1.5pt]
\tikzstyle{blue morphism}=[fill=white, draw={rgb,255: red,15; green,0; blue,150}, shape=rectangle, text={rgb,255: red,15; green,0; blue,150}, tikzit category=blue]
\tikzstyle{red morphism}=[fill=white, draw={rgb,255: red,150; green,0; blue,2}, shape=rectangle, text={rgb,255: red,150; green,0; blue,2}, tikzit category=red]
\tikzstyle{blue state}=[fill=white, draw={rgb,255: red,15; green,0; blue,150}, shape=circle, regular polygon, regular polygon sides=3, minimum width=0.8cm, shape border rotate=180, inner sep=0pt, text={rgb,255: red,15; green,0; blue,150}, tikzit category=blue]
\tikzstyle{blue node}=[fill={rgb,255: red,15; green,0; blue,150}, draw={rgb,255: red,15; green,0; blue,150}, shape=circle, tikzit category=blue, inner sep=1.5pt]
\tikzstyle{blue}=[text={rgb,255: red,15; green,0; blue,150}, tikzit draw={rgb,255: red,191; green,191; blue,191}, tikzit category=blue, tikzit fill=white, inner sep=0mm]
\tikzstyle{blue wide state}=[fill=white, draw={rgb,255: red,15; green,0; blue,150}, text={rgb,255: red,15; green,0; blue,150}, shape=isosceles triangle, minimum width=0.8cm, shape border rotate=270, inner sep=1.4pt, minimum height=0.5cm, isosceles triangle apex angle=80]
\tikzstyle{red node}=[fill={rgb,255: red,150; green,0; blue,2}, draw={rgb,255: red,150; green,0; blue,2}, shape=circle, inner sep=1.5pt]
\tikzstyle{Purple node}=[fill={rgb,255: red,120; green,0; blue,120}, draw={rgb,255: red,120; green,0; blue,120}, text={rgb,255: red,120; green,0; blue,120}, shape=circle, inner sep=1.5pt]
\tikzstyle{red}=[text={rgb,255: red,150; green,0; blue,2}, inner sep=0mm, tikzit fill=white, tikzit draw={rgb,255: red,191; green,191; blue,191}]
\tikzstyle{purple}=[text={rgb,255: red,150; green,0; blue,150}, inner sep=0mm, tikzit fill=white, tikzit draw={rgb,255: red,191; green,191; blue,191}]
\tikzstyle{white morphism}=[fill=white, draw=white, shape=rectangle, tikzit draw={rgb,255: red,139; green,139; blue,139}]
\tikzstyle{leak morphism}=[fill=white, draw={rgb,255: red,120; green,0; blue,85}, shape=rectangle, text={rgb,255: red,120; green,0; blue,85}, tikzit category=leak]
\tikzstyle{leak}=[text={rgb,255: red,120; green,0; blue,85}, inner sep=0mm, tikzit fill=white, tikzit draw={rgb,255: red,191; green,191; blue,191}, tikzit category=leak]
\tikzstyle{leak node}=[fill={rgb,255: red,120; green,0; blue,85}, draw={rgb,255: red,120; green,0; blue,85}, shape=circle, inner sep=1.5pt, tikzit category=leak]
\tikzstyle{horiz state}=[fill=white, draw=black, regular polygon, regular polygon sides=3, minimum width=1cm, shape border rotate=90, inner sep=0pt]

\tikzstyle{arrow}=[->]
\tikzstyle{dashed box}=[-, dashed]
\tikzstyle{blue arrow}=[-, draw={rgb,255: red,15; green,0; blue,150}, tikzit category=blue]
\tikzstyle{red arrow}=[-, draw={rgb,255: red,150; green,0; blue,2}, tikzit category=red]
\tikzstyle{purple arrow}=[->, draw={rgb,255: red,120; green,0; blue,120}, >=stealth, shorten <=2pt, shorten >=2pt]
\tikzstyle{protected purple arrow}=[->, draw={rgb,255: red,120; green,0; blue,120}, >=stealth, shorten <=2pt, shorten >=2pt, preaction={line width=1.8pt, white, draw}]
\tikzstyle{mapsto}=[{|->}]
\tikzstyle{double wire}=[-, double]
\tikzstyle{curly brace}=[-, draw=none, tikzit draw={rgb,255: red,128; green,0; blue,128}]
\tikzstyle{protected}=[-, preaction={line width=1.8pt,white,draw}]
\tikzstyle{leak arrow}=[-, tikzit draw={rgb,255: red,150; green,0; blue,120}]
\tikzstyle{protected leak arrow}=[-, tikzit draw={rgb,255: red,150; green,0; blue,120}]
\tikzstyle{hollow arrow}=[-, very thin, white, preaction={line width=0.7pt,draw={rgb,255: red,120; green,0; blue,85}}, tikzit category=leak, tikzit draw={rgb,255: red,150; green,0; blue,120}]
\tikzstyle{protected hollow arrow}=[-, very thin, white, preaction={line width=0.7pt,draw={rgb,255: red,120; green,0; blue,85},preaction={line width=2.1pt,white,draw}}, tikzit category=leak, tikzit draw={rgb,255: red,150; green,0; blue,120}]
\tikzstyle{over arrow}=[-, black, preaction={draw=white, double}]

\usepackage{enumitem}
	\setitemize{itemsep=0pt,topsep=6pt}
\usepackage[T1,T2A]{fontenc}			
\usepackage{centernot}
\usepackage{xcolor}
\usepackage{amssymb,xparse}		
\usepackage{calc}
\usepackage{geometry}
\usepackage{authblk}	

\definecolor{myurlcolor}{rgb}{0,0,0.3}
\definecolor{mycitecolor}{rgb}{0,0.3,0}
\definecolor{myrefcolor}{rgb}{0.3,0,0}
\usepackage[pagebackref,draft=false]{hyperref}
\hypersetup{
	colorlinks,
	linkcolor=myrefcolor,
	citecolor=mycitecolor,
	urlcolor=myurlcolor}

\usepackage[noabbrev]{cleveref}
	\graphicspath{ {./figures/} }

\numberwithin{equation}{section}

\newtheorem{theorem}{Theorem}[section]
\newtheorem*{theorem*}{Theorem}
\newtheorem{proposition}[theorem]{Proposition}
\newtheorem{lemma}[theorem]{Lemma}
\newtheorem{corollary}[theorem]{Corollary}
\newtheorem{definition}[theorem]{Definition}
\newtheorem{assumption}[theorem]{Assumption}

\newtheorem{notation}[theorem]{Notation}
\theoremstyle{definition}
\newtheorem{example}[theorem]{Example}
\newtheorem{remark}[theorem]{Remark}

\newcommand{\into}{\hookrightarrow}

\newcommand{\ph}{\mathord{\rule[-0.05em]{0.6em}{0.05em}}}		
\newcommand{\phsm}{\mathord{\rule[-0.035em]{0.4em}{0.035em}}}		

\newcommand{\cat}[1]{{\mathsf{#1}}} 
\newcommand{\op}{\mathrm{op}}
\newcommand{\cD}{\mathsf{D}}		

\newcommand{\Kl}{\mathsf{Kl}}		

\newcommand{\id}{\mathrm{id}} 		

\newcommand{\N}{\mathbb{N}}
\newcommand{\R}{\mathbb{R}}

\newcommand{\B}{W}								
\definecolor{parametrized}{RGB}{15,0,150}
\newcommand{\param}[1]{{\color{parametrized}#1}}
\newcommand{\tensor}{\mathbin{\otimes}}

\tikzset{pullback/.style={minimum size=1.2ex,path picture={	
			\draw[opacity=1,black,-,#1] (-0.5ex,-0.5ex) -- (0.5ex,-0.5ex) -- (0.5ex,0.5ex);%
}}}

\newcommand{\cC}{\mathsf{C}}		
\renewcommand{\det}{\mathrm{det}}	
\newcommand{\ic}{\mathrm{ic}}
\newcommand{\samp}{\mathsf{samp}}	
\newcommand{\thunk}{\mathsf{thunk}}	
\newcommand{\force}{\mathsf{force}}	
\newcommand{\prsamp}{\mathsf{prsamp}} 

\newcommand{\Rel}{\mathsf{Rel}}

\newcommand{\FinStoch}{\mathsf{FinStoch}}
\newcommand{\Stoch}{\mathsf{Stoch}}
\newcommand{\BorelStoch}{\mathsf{BorelStoch}}
\newcommand{\as}[1]{
		\def\relstate{#1}%
		\ifx\relstate\empty
		  \text{a.s.}%
		\else
		  {#1\text{-a.s.}}%
		\fi
	}
\newcommand{\ase}[1]{=_{#1\text{-a.s.}}}	
\newcommand{\asinf}[1]{\succeq_{#1\text{-a.s.}}}		
\newcommand{\Bayesinf}{\succeq_{\mathrm{Bayes}}}	

\newcommand{\pass}{\checkmark}
\newcommand{\fail}{\times}
\newcommand{\safe}{\mathsf{s}}
\newcommand{\faulty}{\mathsf{f}}

\DeclareMathOperator{\cop}{copy}
\newcommand{\discard}{\mathrm{del}}

	\providecommand{\given}{}			
	\newcommand{\SetSymbol}[1][]{%
		\nonscript\;\,#1\vert
		\allowbreak
		\nonscript\;\,
		\mathopen{}
	}
	\DeclarePairedDelimiterX{\Set}[1]{\{}{\}}{%
		\renewcommand{\given}{\SetSymbol[\delimsize]}
		#1
	}
	\makeatletter
		\let\oldSet\Set
		\def\Set{\@ifstar{\oldSet}{\oldSet*}}
	\makeatother

\newcommand{\setmuskip}[2]{#1=#2\relax}
\setmuskip{\medmuskip}{4mu plus 2mu minus 2mu}

	\newcommand{\hyphenationsetting}{%
		\emergencystretch=0pt	
		\tolerance=2000			
		\pretolerance=1000		
		\righthyphenmin=4		
		\lefthyphenmin=4			
	}
	\hyphenationsetting

\title{Representable Markov Categories and\\ Comparison of Statistical Experiments\\ in Categorical Probability}

\author[1]{Tobias Fritz\thanks{tobias.fritz@uibk.ac.at}}
\author[2,3]{Tom{\'a}{\v{s}} Gonda\thanks{tomas.gonda@uibk.ac.at}}
\author[4]{Paolo Perrone\thanks{paolo.perrone@cs.ox.ac.uk}}
\author[5]{Eigil Fjeldgren Rischel\thanks{eigil.rischel@strath.ac.uk}}

\affil[1]{Department of Mathematics, University of Innsbruck, Austria}
\affil[2]{Perimeter Institute for Theoretical Physics, Waterloo ON, Canada}
\affil[3]{School of Physics and Astronomy, University of Waterloo, Canada}
\affil[4]{Massachusetts Institute of Technology, Cambridge MA, U.S.A.}
\affil[5]{University of Strathclyde, Glasgow, Scotland}

\begin{document}

\newgeometry{top=2cm,bottom=2cm}
\maketitle
\thispagestyle{empty}

\begin{abstract}
		\emph{Markov categories} are a recent categorical approach to the mathematical foundations of probability and statistics. 
		Here, this approach is advanced by stating and proving equivalent conditions for second-order stochastic dominance, a widely used way of comparing probability distributions by their spread. 
		Furthermore, we lay the foundation for the theory of comparing statistical experiments within Markov categories by stating and proving the classical Blackwell--Sherman--Stein Theorem.
		Our version not only offers new insight into the proof, but its abstract nature also makes the result more general, automatically specializing to the standard Blackwell--Sherman--Stein Theorem in measure-theoretic probability as well as a Bayesian version that involves prior-dependent garbling.
		Along the way, we define and characterize \emph{representable} Markov categories, within which one can talk about Markov kernels to or from spaces of distributions.
		We do so by exploring the relation between Markov categories and Kleisli categories of probability monads.
\end{abstract}

\restoregeometry

\newpage
\tableofcontents

\section{Introduction}

	Traditionally, the foundations of mathematical statistics are rooted in measure theory and measure-theoretic probability.
	More generally, mathematical statistics and probability theory are typically considered as mathematical subjects of a clearly analytical nature. 
	While this has worked well in practice, it is also often the case in mathematics that higher abstraction leads ultimately to deeper understanding, greater generality and ultimately facilitates the development of results and methods of greater complexity.

	This is what the growing field of categorical probability attempts to do by developing a category-theoretical foundation for probability theory and mathematical statistics. 
	A promising approach is provided by \emph{Markov categories} which, in line with categorical thinking, focuses on the morphisms involved in probabilistic reasoning, namely stochastic maps (or Markov kernels). 
	There is growing evidence that Markov categories can serve both as a categorical foundation for, as well as a generalization of, ordinary measure-theoretic probability theory.
	Indeed, similar to how a computer can be programmed either in terms of low-level machine code or in a more accessible and hardware-independent abstract language, it seems to be the case that probability theory likewise can be practiced either in concrete analytical terms based on Kolmogorov's axioms, or in a more abstract \emph{synthetic} form based on the structural axioms of Markov categories.
		
	More specifically, Markov categories allow one to study and make use of:
	\begin{itemize}
		\item Bayes' theorem and Bayesian updating: This was first considered by Golubtsov in~\cite{golubtsov2002kleisli} and rediscovered recently by Cho and Jacobs~\cite{chojacobs2019strings}, with further results on the dagger functor structure of Bayesian inversion in the first named author's~\cite{fritz2019synthetic}.
		\item Conditional independence: This was also defined within this framework by Cho and Jacobs~\cite{chojacobs2019strings}, and more generally in~\cite[Section~12]{fritz2019synthetic}.
		\item Almost sure equality: Again, first done by Cho and Jacobs~\cite{chojacobs2019strings} and then generalized and developed further in~\cite[Section~13]{fritz2019synthetic}.
		\item Sufficient statistics: Some of the basic theorems on sufficient statistics were proven abstractly in~\cite[Sections~14--16]{fritz2019synthetic}.
		\item Kolmogorov extension theorem and 0/1-laws: The \emph{Kolmogorov products} developed by the first-named and last-named author, which arise as infinite products in Markov categories formalizing the Kolmogorov extension theorem, have facilitated synthetic proofs of the classical 0/1-laws of Kolmogorov and Hewitt--Savage~\cite{fritzrischel2019zeroone}.
		\item patterson2020models has developed an algebraic approach to statistical models, drawing and exploiting relations to categorical logic~\cite{patterson2020models}.
	\end{itemize}
	Of course, this only lists those aspects of probability theory and statistics which have been developed synthetically up to the present time and to our knowledge.
	The present paper has two goals: first, to continue the development of the general categorical theory; and second, to add two more items to the above list, namely second-order stochastic dominance and the classical Blackwell--Sherman--Stein Theorem on the comparison of statistical experiments. 
	We now summarize our results on both of these goals, which are related through the latter applications drawing on the former categorical developments.

\paragraph*{Outline and results.}

	In practice, Markov categories often arise as Kleisli categories of affine symmetric monoidal monads. 
	For example, this is the case for $\BorelStoch$, the category of standard Borel spaces and measurable Markov kernels, or equivalently the Kleisli category of the Giry monad on the category of standard Borel spaces and measurable maps.\footnote{Or yet equivalently Polish spaces and measurable maps.} 
	In \Cref{sec:Rep_Markov}, we clarify the relation between Markov categories and Kleisli categories of this type, namely Kleisli categories of affine symmetric monoidal monads on categories with finite products.
	We find that for a Markov category $\cC$, the question of whether $\cC$ arises as a Kleisli category like this is closely linked to the existence of a right adjoint to the inclusion functor $\cC_{\det} \hookrightarrow \cC$, where $\cC_{\det}$ is the cartesian monoidal subcategory of deterministic morphisms in $\cC$; for if $\cC$ is supposed to be the Kleisli category of a monad on $\cC_{\det}$, then this right adjoint must exist for purely formal reasons.
	The existence of such $P \colon \cC \to \cC_\det$ amounts to the natural bijection 
	\begin{equation}
		\cC_{\det}(A, PX) \cong \cC(A, X),
	\end{equation}
	which we interpret as the existence of a \emph{distribution functor} that, for $A = I$, identifies the deterministic morphisms $I \to PX$ with the ``distributions'' over $X$ (morphisms $I \to X$).
	More generally, (not necessarily deterministic) morphisms $A \to X$ can be thought of as being ``classified'' by deterministic morphisms $A \to PX$.
	Not every Markov category has such a distribution functor. 
	For example, the category of finite sets and stochastic matrices $\FinStoch$ does not, since $\FinStoch(A,X)$ is generically infinite, while instead its putative counterpart $\FinStoch_{\det}(A, PX)$ would necessarily have to be finite.

	Our results on the connection between Markov categories and Kleisli categories are then as follows:
	\begin{itemize}
		\item We prove that if $P$ is an affine symmetric monoidal monad on a cartesian monoidal category $\cD$, and $P$ satisfies a certain pullback condition, then the Kleisli category $\Kl(P)$ is a Markov category such that the subcategory of deterministic morphisms $\Kl(P)_{\det}$ is exactly the original category $\cD$.
		\item Conversely, if a Markov category $\cC$ has a distribution functor $P$, meaning a right adjoint for the inclusion $\cC_{\det} \hookrightarrow \cC$, then the induced monad on $\cC_\det$ satisfies the pullback condition, and $\cC$ is isomorphic to the Kleisli category of $P$  (\Cref{markov_as_kleisli}).\footnote{A similar reconstruction of strong monads from their Kleisli categories seems to be known~\cite{callbyvalue}. Nevertheless, our result is not an immediate consequence of this construction.}
	\end{itemize}
	We end \Cref{sec:Rep_Markov} by studying the interaction between the distribution functor $P$ and the notion of almost sure equality in $\cC$.
	If these are compatible in a suitable sense, then we say that $\cC$ is \emph{\as{}-compatibly representable}.
	Distribution functors and \as{}-compatible representability will then play a central role in the subsequent two sections that focus on applications of the general theory.

	In \Cref{sec:2ndorder}, we provide a categorical description and generalization of \emph{second-order stochastic dominance}, which is a way of comparing probability distributions with respect to how ``spread out'' they are. 
	This notion also appears in Blackwell--Sherman--Stein (BSS) theorem, a classical and widely used fundamental result that connects it to the question of comparing statistical experiments in terms of their informativeness about the tested hypotheses.

	In \Cref{sec:blackwell}, we introduce the informativeness preorder in Markov categories, and prove \Cref{thm:informativeness_vs_sufficiency} that characterizes it in terms of the notions of sufficient statistics and conditional independence.
	Most of this section, however, is devoted to a categorical version of the BSS theorem. 
	In fact, we have a few variations thereof. 
	The closest to the standard version that concerns a \emph{discrete} parameter space is \Cref{thm:dilation_criterion_discrete}, but it more generally applies to any \as{}-compatibly representable Markov category other than $\BorelStoch$.
	In our presentation, this result arises as a corollary of \Cref{thm:dilation_criterion} for more general parameter spaces, which considers a fixed prior distribution and compares experiments with respect to whether they are ``almost surely more informative''.
	Concretely, in the context of standard Borel spaces, \Cref{thm:dilation_criterion} says the following.
	\begin{theorem*}
		Let $X$, $Y$ and $\Theta$ be standard Borel spaces, and let $(f_\theta)_{\theta \in \Theta}$ and $(g_\theta)_{\theta \in \Theta}$ be families of probability measures on $X$ and $Y$ respectively, parametrized measurably in $\theta$. Let $m$ be a probability measure on $\Theta$. Then the following are equivalent:
		\begin{enumerate}
			\item There is a Markov kernel $c \colon X \to Y$ such that $g_\theta = c f_\theta$ holds for $m$-almost all~$\theta$.
			\item The standard measures\footnotemark{} $\hat{f}_m$ and $\hat{g}_m$ (probability measures on $P\Theta$---the space of probability measures) are such that $\hat{g}_m$ second-order dominates $\hat{f}_m$.
		\end{enumerate}
	\end{theorem*}
	\footnotetext{Standard measures have been introduced in \cite{blackwell1951comparison}.
		Here, we provide a synthetic definition in \Cref{sec:BSS}.}%
	In this formulation of the BSS Theorem, we do not need to assume that the parameter space $\Theta$ be finite or even countable.
	
	We then present a completely prior-independent version of the BSS theorem in \Cref{sec:blackwell3}. 
	This result avoids the need for a prior by effectively considering all priors at once.
	In our categorical formulation, it turns out to be a special case of the earlier \Cref{thm:dilation_criterion}; but when instantiated in $\BorelStoch$, we obtain the following statement.
	\begin{theorem*}
		Let $X$, $Y$ and $\Theta$ be standard Borel spaces, and let $(f_\theta)_{\theta \in \Theta}$ and $(g_\theta)_{\theta \in \Theta}$ be families of probability measures on $X$ and $Y$ respectively, parametrized measurably in $\theta$. Then the following are equivalent:
		\begin{enumerate}
			\item There is a family of Markov kernels $(c_m \colon X \to Y)_{m \in P\Theta}$, depending measurably on the prior $m$, such that $g_\theta = c f_\theta$ holds for $m$-almost all $\theta$ and all $m \in P\Theta$.
			\item The standard measures $\hat{f}_m$ and $\hat{g}_m$ are such that $\hat{g}_m$ second-order dominates $\hat{f}_m$ for every choice of prior $m \in P\Theta$, as witnessed by a family of dilations $(t_m)_{m \in P\Theta}$ that depend measurably on $m$.
		\end{enumerate}
	\end{theorem*}
	Moreover, as we show in \Cref{prop:garbling_counterexample}, these conditions are not in general equivalent to $f$ being more informative than $g$ with respect to a \emph{prior-independent} garbling map.

\paragraph*{Outlook.}

	Given the relevance of the theory of comparison of experiments in a wide array of situations, such as hypothesis testing or error correction, proving versions of celebrated results---such as the BSS Theorem---in the abstract context of Markov categories leads to a greater level of generality which has the potential for new domains of applications.
	The understanding of these results in a synthetic way also has the potential to overcome some of the limitations of the standard approaches, such as the discreteness of the parameter spaces involved.
	
	With the recent development of quantum Markov categories~\cite{parzygnat2020inverses}, it is conceivable that one could obtain a synthetic version of the quantum BSS Theorem \cite{buscemi2012comparison} and related results, with potential applications to quantum hypothesis testing or quantum error correction.
	
	Finally, the categorical approach also lends itself to the considerations of variants of the theory in which additional restrictions are placed on the garbling maps.
	For example, such variations can be studied under the hood of resource theories of distinguishability as introduced in \cite[Appendix C]{gonda2019monotones}.
	Many interesting restrictions arise from requiring equivariance of the garbling maps with respect to group actions.
	Others include adaptive garbling maps or garbling via independent action of multiple agents, both of which are considered in \cite{de2018blackwell}.
	Although we have not done this yet, it should be straightforward to instantiate \Cref{thm:dilation_criterion} and \Cref{thm:dilation_criterion2} in suitable categories, so as to obtain measure-theoretic BSS theorems which apply in such contexts.

\paragraph*{Acknowledgments.}

	We thank Robert Furber for helpful feedback on measure-theoretic aspects, Luciano Pomatto for helpful feedback on a draft, Jean-Simon Pacaud Lemay for pointers to the literature, and an anonymous referee for additional detailed feedback on an earlier version.
	Research for the first author is supported by FWF (Austrian Science Fund) P 35992-N.
	Research for the third author is funded by AFOSR grants FA9550-19-1-0113 and FA9550-17-1-0058.
	Research for the second author is supported by NSERC Discovery grant RGPIN 2017-04383, and by the Perimeter Institute for Theoretical Physics.
	Research at Perimeter Institute is supported in part by the Government of Canada through the Department of Innovation, Science and Economic Development Canada and by the Province of Ontario through the Ministry of Colleges and Universities.

\section{Markov Categories}\label{sec:Markov_cat}

\subsection{Definition of Markov Categories and Basic Theory}

	We now recall the definition of Markov category. 
	As far as we know, it was first proposed by Golubtsov as \emph{category of information transformers} in slightly different form~\cite{golubtsov2002kleisli},
	used implicitly in Fong's work on Bayesian networks~\cite{fong2012causaltheories},
	and rediscovered recently by Cho and Jacobs as \emph{affine CD-categories}~\cite{chojacobs2019strings}.
	The simpler term \emph{Markov category} was subsequently coined in~\cite{fritz2019synthetic}, based on the idea that Markov categories are abstract generalizations of the category of Markov kernels.

	\begin{definition}\label{markov_cat}
		A \emph{Markov category} $\cC$ is a semicartesian\footnote{Recall that this means that the monoidal unit object $I$ is terminal in $\cC$, among several equivalent characterizations; see~\cite[Theorem~3.5]{gerhold2022categorial}.} symmetric monoidal category where every object $X \in \cC$ is equipped with a distinguished morphism
		\begin{equation}
			\label{comultiplication}
			\input{comultiplication.tikz}
		\end{equation}
		which, together with the unique morphism $\discard_{X} \colon X \to I$, makes $X$ into a commutative comonoid, and such that
		\begin{equation}
			\label{multiplicativity}
			\input{multiplicativity.tikz}
		\end{equation}
		for all $X, Y \in \cC$.
	\end{definition}
	
	Throughout this manuscript, $\cC$ denotes a Markov category. 

	Among the prototypical examples of a Markov category is $\BorelStoch$, the category of standard Borel spaces and measurable Markov kernels. 
	A more basic example is $\FinStoch$, the category of finite sets and stochastic matrices. 
	In both cases, the comultiplications $\cop_X \colon X \to X \otimes X$ are given by the diagonals $x \mapsto \delta_{(x,x)}$, assigning to every element $x \in X$ the Dirac delta distribution $\delta_{(x,x)}$; this is the stochastic way to talk about copying. 
	Other examples of Markov categories can be obtained from categories of relations, such as $\Rel$, by restricting to relations $R \colon X \rightsquigarrow Y$ which have the property that for every $x \in X$ there is $y \in Y$ with $x R y$; this is the relational analogue of the normalization of probability.
	This results in a Markov category with respect to the usual cartesian product as monoidal structure, and the copy maps are again given by the obvious diagonals. 
	Another interesting class of examples arises by noting that diagram categories of Markov categories are again Markov categories (when suitably defined~\cite[Section~7]{fritz2019synthetic}), and we expect that this can be used in future work as a basis for a synthetic theory of stochastic processes.
	
\begin{definition}\label{def:deterministic}
	A morphism $f \colon X \to Y$ in $\cC$ is \emph{deterministic} if it respects the copy maps:
	\begin{equation}
		\input{multiplication_natural.tikz}
	\end{equation}
	The subcategory of $\cC$ that consists of its deterministic morphisms is denoted by $\cC_{\det}$.
\end{definition}

	This type of condition goes back to the seminal paper of Carboni and Walters on cartesian bicategories~\cite{carboniwalters1987bicartesian}. 
	Intuitively, it means that applying $f$ to two independent copies of its input is guaranteed to result in the same pair of output values than applying $f$ directly to the input and copying its output. $\cC$ is a cartesian monoidal category with respect to the monoidal structure inherited from $\cC$, and all structure morphisms of $\cC$, including the copy maps, are in $\cC_{\det}$~\cite[Remark~10.13]{fritz2019synthetic}.

	Other key notions within Markov categories that we use in \Cref{sec:2ndorder,sec:blackwell} include conditionals, Bayesian inverses, almost sure equality, and domination (in the sense of absolute continuity).
	All but the last of these notions have been introduced in earlier works \cite{chojacobs2019strings,fritz2019synthetic}.
	We now recall their definitions.
	\begin{definition}\label{def:conditional}
		Given $f \colon A \to X \tensor Y$ in $\cC$, a morphism \mbox{$f_{|X} \colon X \tensor A \to Y$} in $\cC$ is called a \emph{conditional} of $f$ with respect to $X$ if the equation
		\begin{equation}\label{eq:conditional}
			\input{conditional.tikz}
		\end{equation}
		holds.
		We say that $\cC$ \emph{has conditionals} provided that such a conditional exists for all objects $A,X,Y \in \cC$ and all $f \colon A \to X \otimes Y$ in $\cC$.
	\end{definition}

	We can also consider conditionals of $f \colon A \to X \otimes Y$ with respect to $Y$, which are defined in the analogous way. Using the symmetry of $\cC$ shows that these automatically exist if $\cC$ has conditionals.

	\begin{example}
		\label{ex:borelstoch_has_conds}
		$\BorelStoch$ has conditionals~\cite[Example~11.7]{fritz2019synthetic}. As far as we know at the moment, the earliest reference for this measure-theoretic fact is in Kallenberg's textbook on random measures~\cite[Theorem~1.25]{kallenberg2017randommeasures}.
	\end{example}
	
	\begin{definition}\label{def:Bayesian_inverse}
		Given two morphisms $m \colon I \to A$ and $f \colon A \to X$, a \emph{Bayesian inverse} of $f$ with respect to (the prior) $m$ is a conditional of
		\begin{equation}\label{eq:fm}
			\input{fm.tikz}
		\end{equation}
		with respect to $X$.
	\end{definition}
	The choice of a prior is often clear from context, so we denote a Bayesian inverse of $f$ simply by $f^{\dagger} \colon X \to A$ with the dependence on $m$ left implicit. Thus a Bayesian inverse $f^{\dagger}$ is defined to be a morphism satisfying the equation:
	\begin{equation}\label{eq:bayesian_inverse}
		\input{bayesian_inverse.tikz}
	\end{equation}
	Even though conditionals and Bayesian inverses are generally not unique when they exist, it is clear from the definition that they \emph{are} unique up to almost sure equality \cite[Proposition~13.6]{fritz2019synthetic}, which in general is defined as follows.
	\begin{definition}\label{def:ase}
		Given any morphism $h \colon A \to X$, we say that any two parallel $f, g \colon X \to Y$ are $h$-almost surely equal, denoted by $f \ase{h} g$, if we have
		\begin{equation}\label{eq:ase}
			\input{ase.tikz}
		\end{equation}
	\end{definition}
	\begin{example}\label{ex:bs_ase}
		In the context of $\BorelStoch$, \Cref{def:ase} recovers the expected notion of equality almost surely as has been shown in \cite[Proposition 5.4]{chojacobs2019strings}.
		In particular, given Markov kernels $f,g \colon X \to Y$ and $\nu \colon I \to X$, the relation $f \ase{\nu} g$ means exactly that for all $S \in \Sigma_X$ and $T \in \Sigma_Y$, we have
		\begin{equation}\label{eq:bs_ase}
			\int_S f(T|x) \, \nu(dx) = \int_S g(T|x) \, \nu(dx),
		\end{equation}
		or equivalently that the integrands $f(T|\ph)$ and $g(T|\ph)$ are $\nu$-almost everywhere equal for all $T$.
	\end{example}
	
	The following notion of measure domination is new in the context of Markov categories.
	We consider this definition tentative for the moment; we will be using it in this form in the present paper, but note that we may adopt a different variant of this definition in future work.
	\begin{definition}\label{def:domination}
		Given two morphisms $\mu, \nu \colon I \to X$, we say that $\mu$ is \emph{absolutely continuous} with respect to $\nu$, denoted $\nu \gg \mu$ or $\mu \ll \nu$, if for all objects $Y$ and all morphisms $f,g \colon X \to Y$ we have
		\begin{equation}\label{eq:domination}
			f \ase{\nu} g  \quad \implies \quad f \ase{\mu} g.
		\end{equation}	
	\end{definition}
	
	\begin{example}
		\label{ex:bs_domination}
		In $\BorelStoch$, \Cref{def:domination} recovers the standard notion of domination of probability measures (also known as absolute continuity preorder), given by the condition that for all measurable sets $S \in \Sigma_X$, we have
		\begin{equation}
			\label{standard_dominance}
			\nu(S) = 0 \quad \implies \quad \mu(S) = 0.
		\end{equation}
		
		To prove that this is indeed the case, suppose first that condition~\eqref{standard_dominance} holds. 
		One can then replace $\nu$ with $\mu$ in \cref{eq:bs_ase}, so that $f \ase{\nu} g$ indeed implies $f \ase{\mu} g$ as necessary to conclude $\nu \gg \mu$ according to \Cref{def:domination}.

		In the converse direction, suppose that $\nu \gg \mu$ holds in the sense of \Cref{def:domination}, and that $\nu(S) = 0$ for some $S \in \Sigma_X$. 
		Consider $f$ and $g$ to be the Markov kernels associated to the measurable functions $1_S \colon X \to \{0,1\}$ and $X \to \{0,1\}, \: x \mapsto 0$ respectively. 
		Then we have $f \ase{\nu} g$ by $\nu(S) = 0$. 
		However, together with $\nu \gg \mu$ this gives $f \ase{\mu} g$, which is just a different way to write $\mu(S) = 0$ given our choice of $f$ and $g$.
	\end{example}

	\subsection{Parametric Markov Categories}\label{sec:parametric}

		In order to demonstrate the power of the synthetic treatment of the notions of second-order stochastic dominance and comparison of statistical experiments later, we use the following new class of Markov categories throughout this paper.
		
		Given any Markov category $\cC$ and any object $\B \in \cC$, we now define a new Markov category $\cC_{\B}$ which we call the \emph{Markov category parametrized by $\B$}, or simply a \emph{parametric Markov category} when referring to no particular choice of $\B$.
		This is essentially a known construction for symmetric monoidal categories that has been called \emph{comonoid indexing}~\cite{Hyland1999games}.

		The objects of $\cC_{\B}$ coincide with those of $\cC$, and its morphisms $\param{A \to X}$ are defined to be precisely the morphisms $\B \otimes A \to X$ in $\cC$, that is
		\begin{equation}
			\cC_{\B}(A,X) \coloneqq \cC(\B \otimes A, X).
		\end{equation}
		We think of the object $\B$ as playing the role of a ``parameter space'' which indexes a family of morphisms $A \to X$.
		In order to distinguish notationally between morphisms $\param{A \to X}$ in $\cC_\B$ and their representatives $\B \otimes A \to X$ in $\cC$, we use blue colored text and diagrams whenever the former representation is used, but otherwise use the same symbol to denote the two.
		The composition of morphisms in $\cC_\B$ is defined by distributing the parameter in $\B$ via the copy map $\cop_\B$:
		\begin{equation}\label{eq:param_composition}
			\input{param_composition.tikz}
		\end{equation}
		The tensor product of morphisms in $\cC_\B$ is likewise defined by supplying copies of $\B$ to the respective morphisms,
		\begin{equation}\label{eq:param_tensor}
			\input{param_tensor.tikz}
		\end{equation}
		and with the monoidal structure morphisms being precisely those of $\cC$ itself.
		The discarding operation $\discard_X$ in $\cC_\B$ just consists of discarding both $\B$ and $X$.
		Finally, the copying in $\cC_\B$ also discards the parameter,
		\begin{equation}\label{eq:param_copy}
			\input{param_copy.tikz}
		\end{equation}
		It is then straightforward to verify that $\cC_\B$ is indeed also a Markov category.

		We can alternatively think of $\cC_\B$ as the co-Kleisli category of the reader comonad\footnote{Depending on the literature, this is also known as ``writer comonad'', since its underlying functor is the same as the writer monad in case $\B$ is a monoid object, as well as ``product comonad''.} $\B \otimes \ph$ on $\cC$ (see for example~\cite[Section~5.3]{perrone2019notes}).
		Note that, while the reader comonad is usually defined on cartesian monoidal categories, the only property of cartesian monoidal categories that is actually used in the definition is that the object $\B$ has a comonoid structure, and thus this co-Kleisli category still makes sense in our context.
		
		\begin{lemma}
			If $\cC$ has conditionals, then so does every parametric Markov category $\cC_\B$.
		\end{lemma}
		\begin{proof}
			If $\param{f \colon A \to X \otimes Y}$ is a morphism in $\cC_\B$ represented by $f \colon \B \otimes A \to X \otimes Y$ in $\cC$, then every conditional $f_{|X} \colon X \otimes \B \otimes A \to Y$ of $f$ with respect to $X$ represents a conditional $\param{f_{|X}}$ of $\param{f}$ in $\cC_\B$ upon permuting its input factors to $\B \otimes (X \otimes A)$.
		\end{proof}
			
\section{Representable Markov Categories}\label{sec:Rep_Markov}

\subsection{Kleisli Categories as Markov Categories}\label{sec:Kleisli_is_Markov}

	It was argued by Kock~\cite{kock2012distributions} that affine commutative monads provide a convenient categorical framework for theories of distributions.
	The following result, which is a special case of~\cite[Proposition~3.1]{fritz2019synthetic} gives one direction of the connection between this monadic approach and Markov categories.
	
	Recall first that a monad $P$ on a category with a terminal object $I$ is called \emph{affine} if $P(I) \cong I$ holds.
	Since commutative monads and symmetric monoidal monads are equivalent concepts~\cite[Proposition~6.3.5]{brandenburg}, the following result can be viewed as taking a variant of Kock's framework as its starting point.
	
	Note that term ``commutative monad'' is more commonly used than the equivalent notion of a ``symmetric monoidal monad'', especially in the computer science literature.
	However, we prefer working with the latter because its monoidal structure maps given in \eqref{eq:lax_monoidal_structure} have a clear probabilistic interpretation.
	Intuitively, if $\mu \in PX$ and $\nu \in PY$ are probability distributions, then $\nabla(\mu, \nu) \in P(X \times Y)$ can be thought of as the corresponding product distribution (see \cref{eq:product_dist}). 
	
	\begin{proposition}
		\label{kleisli_to_markov}
		Let $\cD$ be a cartesian monoidal category, and let $(P,E,\delta)$ be an affine symmetric monoidal monad on $\cD$ with unit $\delta$, multiplication $E$, and monoidal structure maps
		\begin{equation}
			\label{eq:lax_monoidal_structure}
			\nabla \colon P(\ph) \times P(\ph) \to P(\ph \times \ph).		
		\end{equation}
		Then the Kleisli category $\Kl(P)$ is a Markov category with respect to the following pieces of structure:
		\begin{itemize}
			\item The monoidal structure on objects is given by products in $\cD$, and the monoidal product of Kleisli morphisms $f \colon A \to PX$ and $g \colon B \to PY$ represented by the composite
				\[
					\begin{tikzcd}
						A \times B \ar{r}{f \times g}	& PX \times PY \ar{r}{\nabla}	& P(X \times Y),
					\end{tikzcd}
				\]
			\item The copy maps $\cop_X$ are represented by the overall composite of the diagram
				\begin{equation}
					\label{kleisli_copy}
					\begin{tikzcd}
						X \ar{r}{\delta} \ar[swap]{d}{(\id,\id)}			& PX \ar{d}{(\id,\id)}	\\
						X \times X \ar{r}{\delta \times \delta}	\ar[swap]{dr}{\delta}	& PX \times PX \ar{d}{\nabla}	\\
														& P(X \times X)
					\end{tikzcd}
				\end{equation}
		\end{itemize}
	\end{proposition}

	Note that the upper square in \cref{kleisli_copy} commutes trivially, while the lower triangle commutes as one of the defining properties of monoidal monads. 
	
\begin{example}
	This construction reproduces $\BorelStoch$ as the Kleisli category of the Giry monad on the category of standard Borel spaces and measurable maps; the definition of the copy maps reproduces exactly the maps $x \mapsto \delta_{(x,x)}$ described above.
\end{example}

\begin{example}\label{ex:D_R}
	Let $(R,+,\cdot,0,1)$ be a commutative semiring, i.e.\ a set $R$ equipped with algebraic structure like that of a commutative ring except for the assumption of additive inverses. Then $R$ induces an affine symmetric monoidal monad $D_R$ on $\cat{Set}$, given by the $R$-linear combinations monad together with a normalization constraint.
	This is spelled out, for example, in \cite[Section~5.1]{coumans2013scalars}, which we recall here.
 
	For each set $X$, denote by $D_R X$ the set of functions $p \colon X\to R$ which are nonzero on finitely many elements, and such that the normalization constraint
	\begin{equation}
		\sum_{x\in X} p(x) = 1
	\end{equation}
	holds.
	This sum is well-defined thanks to the fact that it has at most a finite number of nonzero summands,
		which is also the case for all other sums appearing in this example.
	
	For every set function $f \colon X\to Y$, we can construct the corresponding function $D_R f \colon D_R X \to D_R Y$ as follows. 
	Given $p \in D_R X$, we define $(D_R f)(p)$ to be the map 
	\begin{equation}
		y \mapsto \sum_{x\in f^{-1}(y)} p(x) .
	\end{equation}
	This makes $D_R$ into a functor. 
	The unit of the monad has components $\eta \colon X \to D_R X$ that map each $x \in X$ to $\eta(x) \colon X \to R$ defined by
	\begin{equation}
		x' \mapsto
			\begin{cases}
				1 &  \text{if }  x=x',   \\
				0 &  \text{if }  x \ne x' ,
      		\end{cases}
	\end{equation}
	generalizing the Dirac delta distribution to the commutative semiring setting.
	The monad multiplication map $\mu \colon D_R D_R X \to D_R X$ is given by
	\begin{equation}
		\mu(\phi)(x) \coloneqq \sum_{p\in D_R X} \phi(p) \cdot p(x)
	\end{equation}
	for all $\phi\in D_R D_R X$ and $x\in X$, where the product is taken in $R$.
	The monoidal unit map is uniquely determined because $D_R I \cong I$ is the terminal object. 
	Finally, the monoidal multiplication map $\nabla \colon D_R X \times D_R Y \to D_R (X\times Y)$ is given by 
	\begin{equation}\label{eq:product_dist}
		\nabla(p,q) (x,y) \coloneqq p(x) \cdot p(y)
	\end{equation}
	for all $p\in D_R(X)$, $q\in D_R(Y)$, $x\in X$ and $y\in Y$. The commutativity of $R$ is relevant for showing that this lax monoidal structure is symmetric. We leave the detailed verifications to the reader.
 
	Hence we have specified $D_R$ as an affine symmetric monoidal monad on $\cat{Set}$, which we call the \emph{(generalized) distribution monad valued in $R$}.
	By \Cref{kleisli_to_markov}, its Kleisli category is canonically a Markov category.
\end{example}

	Returning to the general theory, we consider the relation between $\cD$ and the subcategory of deterministic morphisms $\Kl(P)_\det$ in the Kleisli category.
	Clearly, the canonical identity-on-objects functor $\cD \to \Kl(P)$ lands in $\Kl(P)_\det$. 
	For particular monads $P$ it often happens that this functor is fully faithful, and hence an isomorphism of categories: 
	The original category $\cD$ is precisely the category of deterministic morphisms. 

	For example, this happens with the Giry monad on standard Borel spaces, for which the Kleisli category is $\BorelStoch$~\cite[Example~10.5]{fritz2019synthetic}.
	On the other hand, it does \emph{not} happen for $\Stoch$ as the Kleisli category of the Giry monad on all measurable spaces: 
	There are $\{0,1\}$-valued probability measures on suitable measurable spaces $(X,\Sigma_X)$ which are not delta measures~\cite[Example~10.4]{fritz2019synthetic}. 
	Some unfolding of the definitions shows that such a measure defines a deterministic morphism $I \to (X,\Sigma_X)$ in $\Stoch$ which does not correspond a measurable map $I \to (X,\Sigma_X)$, since the latter correspond exactly to the delta measures on $X$.

	We now present a general criterion which guarantees that there are no such ``accidental'' deterministic morphisms.
	Intuitively, it states that the delta distributions should be precisely those distributions which are independent of themselves, or equivalently, that they should be the only product measures supported on the diagonal.
	\begin{proposition}
		\label{det_is_nice}
		Let $\cD$ be a cartesian monoidal category.
		Let $(P,E,\delta)$ be an affine symmetric monoidal monad on $\cD$. 
		Then the canonical functor $\cD \to \Kl(P)_\det$ is an isomorphism of categories if and only if the diagram 
		\begin{equation}
			\label{delta_pullback}
			\begin{tikzcd}
				X \ar{r}{\delta} \ar[swap]{d}{(\delta,\delta)} \ar[dr,phantom," "{pullback},very near start]	& PX \ar{d}{P(\id,\id)}	\\
				PX \times PX \ar{r}{\nabla}									& P(X \times X)
			\end{tikzcd}
		\end{equation}
		is a pullback for every $X \in \cD$.
	\end{proposition}

	\begin{proof}
		The monoidal structure map $\nabla$ has a left inverse given by the canonical map
		\[
			\Delta \colon P(X \times X) \to PX \times PX
		\]
		induced from the cartesian monoidal structure of $\cD$ (note that this map corresponds to marginalization in the probability context~\cite{fritz2018bimonoidal}).
		Therefore, $\nabla$ is a monomorphism. 
		Since monomorphisms are stable under pullback, it follows that $\delta \colon X \to PX$ is a monomorphism as well.
		This implies that the canonical functor $\cD \to \Kl(P)_\det$ is faithful.
	
		To prove fullness, let $f \colon A \to PX$ be the representative of a deterministic Kleisli morphism $A \to X$ in the Markov category $\Kl(P)$. Some unfolding of the definitions shows that the determinism assumption amounts exactly to commutativity of the diagram
		\begin{equation}
			\begin{tikzcd}
				A \ar[swap]{d}{(f,f)} \ar{r}{f}  &  PX \ar{d}{P(\id,\id)}  \\
				PX \times PX \ar{r}{\nabla}  &  P(X \times X)
			\end{tikzcd}	
		\end{equation}
		But now the assumption that diagram~\eqref{delta_pullback} is a pullback lets us obtain the dashed arrow $\tilde{f}$ in
		\begin{equation}
		  \label{eq:pback_problem_kleisli}
			\begin{tikzcd}
				A \ar[bend right,swap]{ddr}{(f,f)} \ar[bend left]{drr}{f} \ar[dashed]{dr}{\tilde{f}}  &    &    \\
				  &  X \ar{d}{(\delta,\delta)} \ar{r}{\delta}  &  PX \ar{d}{P(\id,\id)}  \\
				  &  PX \times PX \ar{r}{\nabla}  &  P(X \times X)
			\end{tikzcd}
		\end{equation}
		which is exactly the factorization of $f$ needed to show that it is in the image of $\cD \to \Kl(P)_\det$.
	
		Conversely, suppose that $\cD \to \Kl(P)_{\det}$ is an isomorphism.
		Our goal is now to show the unique existence of the dashed arrow in diagram~\eqref{eq:pback_problem_kleisli}. 
		We observe that the arrow $f \colon A \to PX$ represents an arrow $A \to X$ in $\Kl(P)$. 
		The commutativity of the outer square entails that this arrow is deterministic, so that there is a unique preimage $\bar{f} \colon A \to X$ in $\cD$.
		The condition that $\bar{f}$ is sent to $f$ is precisely the condition that the upper triangle commutes---the lower left triangle then commutes automatically by construction of the arrows.
	\end{proof}

	\begin{example}
		\label{ex:dist_monad_det}
		Consider the distribution monad $D_R$ valued in a commutative semiring $R$ as in \Cref{ex:D_R}. 		
		Then depending on what $R$ is, the diagram~\eqref{delta_pullback} for $P=D_R$ may or may not be a pullback for all sets $X$. For example when $R = \R_+$, we recover the usual distribution monad involving finitely supported probability measures, and \eqref{delta_pullback} is a pullback since every $\{0,1\}$-valued and finitely supported probability measure is a Dirac delta.

		The most trivial examples when \eqref{delta_pullback} is not a pullback occur when $\delta$ does not have monomorphism components.
		For instance, if $R$ is the zero semiring, then the associated distribution monad $D_R$ on $\cat{Set}$ is the terminal monad, since every $D_R X$ is a singleton set containing the unique map $X \to R$. In this case, it is clear that \eqref{delta_pullback} is a pullback only when $X$ itself is a singleton set.
	
		For a less trivial example, namely one in which $\delta \colon X \to PX$ is in fact injective but \eqref{delta_pullback} is still not a pullback, let $P$ be the distribution monad $D_{R \oplus R}$ for any nonzero commutative semiring $R$, where the addition and multiplication in $R \oplus R$ are component-wise. 
		Consider the set $X \coloneqq \{a,b\}$ and the distribution
		\begin{equation}
			s \coloneqq (0,1) \, \delta_a + (1,0) \, \delta_b \: \in \: PX
		\end{equation}
		for $0,1 \in R$.
		Clearly $s$ is not a delta distribution, since the only two delta distributions in $PX$ are $(1,1) \delta_a$ and $(1,1) \delta_b$.
		Nevertheless, both the product distribution $s \otimes s$ and $P(\id,\id)(s)$ are equal to
		\begin{equation}
			(0,1) \, \delta_{(a,a)} + (1,0) \, \delta_{(b,b)} \: \in \: P(X \times X).		
		\end{equation}
		Therefore, thinking of $s$ as a morphism $I \to PX$ in $\cat{Set}$ and using it in place of $f$ in diagram~\eqref{eq:pback_problem_kleisli} proves that diagram~\eqref{delta_pullback} is not a pullback in this case.
		Although $s$ is not a delta measure, $s \colon I \to X$ is a deterministic morphism in $\Kl(D_{R \oplus R})$, correctly capturing the intuition that $s$ does not produce any randomness.
	\end{example}

	A semiring $R$ is \emph{entire} if $R \not\cong 0$ and $R$ has no zero divisors. 
	In contrast to \cref{ex:dist_monad_det}, we now establish entirety as a sufficient condition for the deterministic morphisms in the Kleisli category of $D_R$ to be precisely the ones in the image of the functor $\cat{Set} \to \Kl(D_R)$.
	
	\begin{proposition}
		\label{entire}
		For an entire commutative semiring $R$, the diagram \eqref{delta_pullback} with $P = D_R$ is a pullback for all $X$.
	\end{proposition}

	\begin{proof}
		Since $\delta$ has monomorphism components by $1 \neq 0$ in $R$, it is enough to prove that for every $r_1, r_2, s \in PX$ such that
		\begin{equation}\label{eq:det_condition}
			r_1 \otimes r_2 = P(\id,\id)(s)
		\end{equation}
		holds, we necessarily have $r_1 = r_2 = s = \delta_x$ for some $x \in X$. 
		Equation (\ref{eq:det_condition}) unfolds to
		\begin{equation}
			r_1(x_1) \, r_2(x_2) = 
			\begin{cases}
				s(x_1) & \text{if } x_1 = x_2 \\
				0 & \text{otherwise}
			\end{cases}
		\end{equation}
		for all $x_1, x_2 \in X$. 
		Since $\sum_{x_1} r_1(x_1)$ is equal to $1$ by normalization, there must be an $\tilde{x} \in X$ such that $r_1(\tilde{x}) \neq 0$. 
		We then necessarily have $r_2(x_2) = 0$ for all $x_2 \neq \tilde{x}$, because $R$ is entire. 
		Therefore, $r_2(\tilde{x}) = 1$ by normalization; and applying the same argument the other way around yields the analogous statement for $r_1$, so that $r_1 = r_2 = \delta_{\tilde{x}}$, from which $s = \delta_{\tilde{x}}$ follows as well.
	\end{proof}
	
\subsection{Markov Categories as Kleisli Categories}\label{sec:distribution_functors}

	Many common Markov categories are indeed Kleisli categories of affine symmetric monoidal monads, as per \Cref{kleisli_to_markov}. 
	In this subsection, we prove a partial converse to this result.
	As we will see, the resulting \emph{representable} Markov categories carry additional structure which we put to use in the rest of the paper: 
	For every object $X$, there is a \emph{distribution object} $PX$, to be interpreted as the space of probability measures on the given space $X$.

	But let us start by asking under what conditions a given Markov category $\cC$ arises from the construction of \Cref{kleisli_to_markov}.
	If the monad $P$ on $\cD$ satisfies the assumption that \eqref{delta_pullback} is a pullback, then \Cref{det_is_nice} provides us with the natural bijection
	\begin{equation}
		\label{kleisli_distribution}
		\Kl(P)_\det \bigl(A,PX\bigr) \: \cong \: \Kl(P) \bigl(A,X\bigr),
	\end{equation}
	intuitively stating that Markov kernels $A \to X$ are in bijection with ordinary maps $A \to PX$, where $PX$ is the ``object of distributions'' on the object $X$.

	In particular, $P$ uniquely extends to a right adjoint to the inclusion functor \mbox{$\Kl(P)_\det \hookrightarrow \Kl(P)$}, resulting in a functor $\Kl(P) \to \Kl(P)_\det$ which we also denote $P$ by abuse of notation.
	On a Kleisli morphism represented by $f \colon X \to PY$ in the original category $\cD$, the naturality of \cref{kleisli_distribution} in $X$ shows that this functor acts by assigning to it the corresponding morphism of free $P$-algebras, namely the composite
	\[
		\begin{tikzcd}
			PX \ar{r}{Pf}		& PPY \ar{r}{E}		& PY,
		\end{tikzcd}
	\]
	where $E$ is the monad multiplication.
	In the probability context, the units and counits of the Kleisli adjunction \eqref{kleisli_distribution} instantiate to maps intimately familiar from probability theory. 
	The unit component $A \to PA$ is of course the maps that assigns delta distributions. 
	The counit $PX \to X$ in $\Kl(P)$, which is the Kleisli morphism represented by $\id_{PX} \colon PX \to PX$, has been less commonly considered explicitly.
	It can be thought of as the Markov kernel $PX \to X$ which assigns to every probability distribution $\mu \in PX$ a random element (a ``sample'') of $X$ distributed according to $\mu$. 
	We thus call it the \emph{sampling map} and denote it by $\samp \colon PX \to X$.

	In summary, if $P$ is an affine symmetric monoidal monad satisfying the relevant pullback condition, then we obtain the natural bijection of \eqref{kleisli_distribution}. 
	From right to left, a Markov kernel $A \to X$ can be reinterpreted as a deterministic map $A \to PX$; from left to right, composing a deterministic map $A \to PX$ by sampling from its output distribution produces a Markov kernel $A \to X$. 
	By construction, we have $\samp \circ \delta = \id$, which can be interpreted to mean that sampling from a delta distribution $\delta_x$ for $x \in X$ returns $x$.
	
	Based on these considerations, it is natural to consider bijections of the same type for arbitrary Markov categories now.
	\begin{definition}\label{def:distribution_objects}
	  Let $\cC$ be a Markov category and $X \in \cC$ an object.
	  A \emph{distribution object} for $X$ is an object $PX$ equipped with a morphism ${\samp_X \colon PX \to X}$ so that the induced map
	  	\begin{equation}
			\samp_X \circ \ph  \: \colon \:  \cC_{\det}(A,PX) \to \cC(A,X)
		\end{equation}
	  is a bijection for all $A \in \cC$.
	\end{definition}
	\begin{notation}\label{not:sharp}
		As before, we call $\samp_X$ the \emph{sampling map} and often drop the subscript if no confusion is likely to arise.
		We write
		\begin{equation}
			(\ph)^{\sharp} \colon \cC(A,X) \to \cC_{\det}(A,PX)
		\end{equation}
		for the inverse of $\samp \circ \ph$.
		Using this notation, the abstract version of the delta distribution map can be identified as 
		\begin{equation}
			\delta_X \coloneqq (\id_X)^\sharp,
		\end{equation}
		i.e.\ it is the unique deterministic morphism $\delta \colon X \to PX$ satisfying
		\begin{equation}\label{eq:samp_delta}
			\samp \circ \delta = \id.
		\end{equation}
	\end{notation}
	
	In other words, $PX$ is a distribution object if it represents the hom-functor
	\[
		\cC(\ph,X) \: \colon \: \cC_\det^\op \to \cat{Set}
	\]
	in $\cC_\det$. 
	The distinguished sampling morphism then arises as one represented by $\id_{PX} \colon PX \to PX$. 
	
	Note that the term ``distribution object'' is motivated by the fact that the global elements $I \to X$ in $\cC$, which are the abstract versions of probability distributions on $X$, correspond to the global elements $I \to PX$ in $\cC_\det$.
	
	\begin{lemma}
	  \label{lem:std_adjunction}
		If every $X \in \cC$ has a distribution object $PX$, then the assignment $X \mapsto PX$ is the object part of a functor $P \colon \cC \to \cC_\det$ which is right adjoint to the inclusion $\cC_\det \hookrightarrow \cC$, and with the counit of the adjunction being the transformation whose components are the sampling maps.
	\end{lemma}
	
	\begin{proof}
		This is part of the standard theory of adjunctions.
	\end{proof}

	\begin{definition}\label{def:representable}
		A Markov category is termed \emph{representable} if every object has a distribution object.
		We call the corresponding right adjoint functor $P \colon \cC \to \cC_\det$ the \emph{distribution functor} for $\cC$.
	\end{definition}
	
	Let's now see some properties of representable Markov categories.
	First of all, for any $f \colon A \to X$ in a representable Markov category, its deterministic counterpart $f^\sharp$ from \Cref{not:sharp} is the adjunct of $f$ given by the composite
		\begin{equation}
			A \xrightarrow{\mkern11mu \delta \mkern11mu} PA \xrightarrow{\mkern7mu Pf \mkern7mu} PX.
		\end{equation}
	
	Also, the faithfulness of the left adjoint $\cC_\det \into \cC$ also implies that the unit components $\delta \colon X \to PX$ are all monomorphisms~\cite[Lemma~4.5.13]{riehl2016category}.
	
	\begin{remark}
		\label{rem:delta_nonnatural}
		An important caveat is that $\delta$ is a natural transformation between $\id \colon \cC_\det \to \cC_\det$ and $P \colon \cC_\det \to \cC_\det$, and in particular natural with respect to deterministic morphisms. 
		But $\delta$ is generally \emph{not} natural with respect to non-deterministic morphisms. 
		This is one way in which denoting the two functors $P \colon \cC \to \cC_\det$ and $P \colon \cC_\det \to \cC_\det$ by the same letter may be initially confusing.

		On the other hand, the sampling transformation $\samp$ from $P \colon \cC \to \cC$ to $\id \colon \cC \to \cC$ is natural with respect to all morphisms in $\cC$.
		In particular, the diagram
		\begin{equation}
			\label{samp_associative}
			\begin{tikzcd}[column sep=large]
				PPX \ar{d}[swap]{\samp} \ar{r}{P \samp}	& PX \ar{d}{\samp}	\\
				PX \ar{r}{\samp}			& X
			\end{tikzcd}
		\end{equation}
		commutes for all $X \in \cC$, which amounts to the usual associativity of the monad multiplication.
	\end{remark}
	
	This situation, where $\samp$ is natural but $\delta$ is not, can be captured by the notion of a \emph{thunk--force category}, which can be interpreted as ``a category that looks like the Kleisli category of a monad'' \cite{fuhrmann-direct-models,furhmann-varieties-of-effects}.
	
	\begin{definition}[\cite{fuhrmann-direct-models}]
	 A \emph{thunk--force structure} on a category $\cat{C}$ amounts to 
	 \begin{itemize}
	  \item an endofunctor $L \colon \cat{C}\to\cat{C}$;
	  \item a family of maps $\thunk_X \colon X\to LX$ for each object $X$; and
	  \item a family of maps $\force_X \colon LX\to X$, 
	 \end{itemize}
	 such that 
	 \begin{itemize}
	  \item the maps $\force_X \colon LX\to X$ assemble to a natural transformation $L\Rightarrow\id$;
	  \item the maps $\thunk_X \colon X\to LX$ may not in general assemble to a natural transformation $\id\Rightarrow L$, but the maps $\thunk_{LX} \colon LX\to LLX$ do assemble to a natural transformation $L\Rightarrow LL$; and
	  \item the following diagrams commute.
	  \[
	   \begin{tikzcd}[column sep=1.3cm]
	    A \ar{r}{\thunk_A} \ar{d}[swap]{\thunk_A} & LA \ar{d}{L(\thunk_A)} \\
	    LA \ar{r}[swap]{\thunk_{LA}} & LLA
	   \end{tikzcd}
       \quad
       \begin{tikzcd}[column sep=1.3cm]
        A \ar{dr}[swap]{\id} \ar{r}{\thunk_A} & LA \ar{d}{\force_A} \\
        & A
       \end{tikzcd}
       \quad
       \begin{tikzcd}[column sep=1.3cm]
        LA \ar{dr}[swap]{\id} \ar{r}{\thunk_{LA}} & LLA \ar{d}{L(\force_A)} \\
        & LA
       \end{tikzcd}
	  \]
	 \end{itemize}
	 A category equipped with a thunk--force structure is called a \emph{thunk--force category} or \emph{abstract Kleisli category}.
	\end{definition}
	
	A representable Markov category is a thunk--force category, where the endofunctor $L$ is the distribution functor $P \colon \cat{C}\to\cat{C}$, and the maps $\thunk$ and $\force$ are given by $\delta$ and $\samp$ respectively. See also \cite{moss2022probability}, but keep in mind that in that paper, the name $\samp$ is used for the map $\thunk$ composed with copying. 
    Now, as we saw in \Cref{rem:delta_nonnatural}, $\delta$ may not be natural against non-deterministic morphisms. 
    In the context of thunk--force categories, this idea is captured by the notion of thunkable morphisms.
	
	\begin{definition}[\cite{fuhrmann-direct-models}]
	 A morphism $f \colon X\to Y$ in a thunk--force category $(\cat{C},L,\thunk,\force)$ is called \emph{thunkable} if and only if the following diagram commutes.
	\begin{equation}
		\begin{tikzcd}
			X \ar{r}{f} \ar{d}[swap]{\thunk_X} & Y \ar{d}{\thunk_Y} \\
			LX \ar{r}[swap]{Lf} & LY
		\end{tikzcd}
	\end{equation}
	\end{definition}
	
	It turns out that, for representable Markov categories, this class of morphisms coincides with that of deterministic morphisms.
	
	\begin{proposition}\label{prop:thunk_is_det}
		A morphism $f \colon X \to Y$ in a representable Markov category is deterministic if and only if it is thunkable, i.e.\ if and only if we have
		\begin{equation}\label{eq:natural_delta}
			\delta_Y \circ f = Pf \circ \delta_X.
		\end{equation}
	\end{proposition}
	
    See also \cite[Theorem~3.14]{moss2022probability} for a more general context. 
	
	\begin{proof}
	 The ``only if'' direction was already noted in \Cref{rem:delta_nonnatural}.
	 For the ``if'' part, we now prove that the top face of the following cube commutes.
	\begin{equation}
		\begin{tikzcd}[column sep=small]
			 X && Y \\
			 & {X\otimes X} && {Y\otimes Y} \\
			 PX && PY \\
			 & {PX\otimes PX} && {PY\otimes PY}
			 \arrow["f", from=1-1, to=1-3]
			 \arrow["\delta"{pos=0.8}, from=1-3, to=3-3]
			 \arrow["\delta"', from=1-1, to=3-1]
			 \arrow["Pf"'{pos=0.7}, from=3-1, to=3-3]
			 \arrow["\cop", from=3-3, to=4-4]
			 \arrow["{Pf\otimes Pf}"', from=4-2, to=4-4]
			 \arrow["\cop"', from=3-1, to=4-2]
			 \arrow["\cop", from=1-3, to=2-4]
			 \arrow["\delta\otimes\delta", from=2-4, to=4-4]
			 \arrow["\cop"', from=1-1, to=2-2]
			 \arrow["\delta\otimes\delta"{pos=0.2}, from=2-2, to=4-2, crossing over]
			 \arrow["{f\otimes f}"{pos=0.2}, from=2-2, to=2-4, crossing over]
    		\end{tikzcd}
	\end{equation}
	 Now,
	 \begin{itemize}
	  \item The front and back faces commute by the assumed naturality \cref{eq:natural_delta};
	  \item The two side faces commute since $\delta$ is deterministic;
	  \item The bottom face commutes since $Pf$ is deterministic.
	 \end{itemize}
	 Therefore, the top face commutes after postcomposing with the front right leg $\delta \otimes \delta$.
	 By \cref{eq:samp_delta}, i.e.\ $\samp\circ\delta = \id$, we conclude that the top face of the cube also commutes as such.
	\end{proof}

	Now, if $\cC = \Kl(P)$ is a Markov category arising from the construction of \Cref{kleisli_to_markov} and the monad $P$ satisfies the pullback condition of \Cref{delta_pullback}, then $\cC$ is representable.

	Somewhat conversely, if $\cC$ is a representable Markov category, then the defining adjunction induces a monad on $\cC_\det$. 
	We denote its underlying functor also by $P \colon \cC_\det \to \cC_\det$, since it differs from $P \colon \cC \to \cC_\det$ from \Cref{lem:std_adjunction} merely by restriction to the subcategory $\cC_\det$. 
	This monad has unit $\delta$ and multiplication $P \samp$.
	Indeed, in the probability context, sampling from the ``inner'' distribution of a distribution of distributions returns the expected distribution, which is consistent with the idea that $P\samp$ is the multiplication of a probability monad. 
	In fact, we can also compose $P \colon \cC \to \cC_\det$ with the inclusion functor on the other side, considering $P$ as a functor $\cC \to \cC$ instead. 
	Hence, $P$ comes in three versions which we do not distinguish notationally; we leave it understood that $P$ can act on any morphism of $\cC$ and always returns a deterministic morphism.

	For every representable Markov category $\cC$ with distribution functor $P$, there is a canonical isomorphism $\cC \cong \Kl(P)$.
	This is an instance of the elementary fact that if any identity-on-objects functor $\cD_1 \to \cD_2$ has a right adjoint, then this makes $\cD_2$ canonically isomorphic to the Kleisli category of the induced monad on $\cD_1$.\footnote{We thank Sam Staton for pointing this fact out to us.}
	
	However, the Markov category structure on $\cC$ equips this monad with additional structure and properties. 
	Next, we show that $P$ is an affine symmetric monoidal monad in a canonical way and that it automatically satisfies the pullback condition of \Cref{det_is_nice}. 
	As a consequence, if the right adjoint of $\cC_\det \into \cC$ exists, then the canonical isomorphism of categories $\cC \cong \Kl(P)$ is in fact an isomorphism of Markov categories.

	\begin{proposition}
		\label{prop:monoidal_adjunction}
		Let $\cC$ be a representable Markov category. 
		Then the right adjoint $P \colon \cat{C} \to \cat{C}_\det$ has a canonical symmetric lax monoidal structure which makes the adjunction between $P$ and the inclusion functor $\iota \colon \cat{C}_\det \hookrightarrow \cat{C}$ into a symmetric monoidal adjunction. 
	\end{proposition}
	The proof is best understood as an instance of the general theory of \emph{doctrinal adjunctions} \cite{doctrinaladjunction}.
	\begin{proof}
		Since both composites and monoidal products of deterministic morphisms are again deterministic, and also all monoidal structure isomorphisms are deterministic, we can equip $\cC_\det$ with the monoidal structure induced from $\cC$, and this makes the inclusion functor $\iota \colon \cC_\det\hookrightarrow\cC$ into a strict symmetric monoidal functor by definition.

		By the general theory of doctrinal adjunctions,\footnote{While the paper~\cite{doctrinaladjunction} is not open access, the result we are using appears as Proposition~2.1 on the nLab page \href{https://ncatlab.org/nlab/show/monoidal+adjunction}{ncatlab.org/nlab/show/monoidal+adjunction}.}
		a right adjoint to a strong monoidal functor is canonically lax monoidal, and the structure maps are given as follows:
		\begin{itemize}
			\item For all objects $X$ and $Y$ of $\cat{C}$, the multiplication map $\nabla$ 
			of the functor $P  \colon \cC \to \cC_{\det}$ is given by 
				\begin{equation*}\begin{tikzcd}
						\nabla \: \colon \: PX \otimes PY \ar{r}{\delta}	& P(PX\otimes PY) \ar{rrr}{P(\samp\otimes\samp)}	&&& P(X\otimes Y),
				\end{tikzcd}\end{equation*}
			     which is deterministic due to being a composite of deterministic morphisms. 
			     Naturality of $\nabla$ means that the following diagram ought to commute for all (not necessarily deterministic) morphisms $f \colon X\to Y$ and $g \colon A\to B$:
			     \begin{equation*}
					\begin{tikzcd}[sep=large]
						PX \otimes PA \ar{d}{Pf\otimes Pg} \ar{r}{\delta} & P(PX\otimes PA) \ar{d}{P(Pf\otimes Pg)} \ar{rr}{P(\samp\otimes\samp)} && P(X\otimes A) \ar{d}{P(f\otimes g)} \\
						PY \otimes PB \ar{r}{\delta} & P(PY\otimes PB) \ar{rr}{P(\samp\otimes\samp)} && P(Y\otimes B)
					\end{tikzcd}
				\end{equation*}
				The left square commutes by naturality of $\delta$ with respect to the deterministic morphism $Pf \otimes Pg$ and the right one by naturality of $\samp$, so that $\nabla$ is indeed natural in both arguments.
				This can be interpreted as the fact that processing two independent random variables independently preserves their independence.

				A straightforward but tedious diagrammatic argument, involving the given properties of $\delta$ and $\samp$ including $\samp \circ \delta = \id$, then shows that the relevant associativity condition for $\nabla$ to be a lax monoidal structure holds as well.
				Compatibility with the braiding $X \otimes Y \to Y \otimes X$ is obvious.

			 \item The natural isomorphism 
				\begin{equation*}		
					\cat{C}_\det \bigl( X, PI \bigr) \cong \cat{C}(X, I)
				\end{equation*}
				shows that $PI \cong I$ by the assumed terminality of $I$. 
				The unit $I \to PI$ is thus the unique morphism of this type, and it automatically satisfies the relevant compatibility conditions with the multiplication.
		\end{itemize} 
		
		Hence, the right adjoint $P \colon \cC \to \cC_\det$ is a symmetric lax monoidal functor.
		It remains to be shown that $\delta$ and $\samp$, as unit and counit of the adjunction, are monoidal transformations.

		The fact that $\delta$ is a monoidal natural transformation means that the following diagram
		\begin{equation*}
			\begin{tikzcd}
				X\otimes Y \ar{r}{\delta\otimes\delta} \ar{dr}[swap]{\delta} & PX \otimes PY \ar{d}{\nabla} \\
				& P(X\otimes Y)
			\end{tikzcd}
		\end{equation*}
		commutes.
		This can be interpreted as the fact that products of Dirac deltas are again Dirac deltas. 
		A formal proof follows via a standard naturality argument together with $\samp \circ \delta = \id$.
		
		Dually, the fact that $\samp$ is a monoidal natural transformation means that the diagram
		\begin{equation*}
			\begin{tikzcd}
				PX \otimes PY \ar{r}{\nabla} \ar{dr}[swap]{\samp \otimes \samp} & P(X\otimes Y) \ar{d}{\samp} \\
				& X\otimes Y
			\end{tikzcd}
		\end{equation*}
		commutes.
		This can be interpreted as the fact that sampling from a product distribution is the same as sampling from the two marginals independently, and again follows formally by similar arguments.
	\end{proof}
	
	\begin{remark}\label{strength}
	 The \emph{strength} of the monoidal monad $P$ is given by the deterministic maps $\sigma_{X,Y}: X\otimes PY \to P(X\otimes Y)$, natural in $\cC_{\det}$, given by the composition\footnote{One can equivalently start from a commutative strength and construct the monoidal structure in terms of it, see \cite[Section~6.3]{brandenburg}.}
	 \[
	 \begin{tikzcd}
		 X\otimes PY \ar{r}{\delta\otimes \id} & PX\otimes PY \ar{r}{\nabla} & P(X\otimes Y).
	 \end{tikzcd}
	 \]
	 The strength satisfies the following commutative diagram,
	 \begin{equation}\label{eq:strength_nat}
		\begin{tikzcd}
			X \otimes PY \ar{r}{\sigma} \ar{dr}[swap]{\id \otimes \samp} & P(X\otimes Y) \ar{d}{\samp} \\
			& X\otimes Y
		\end{tikzcd}
	\end{equation}
	 which has a similar, but ``one-sided'', interpretation to the analogous condition for $\nabla$. 
	 Note that, since the unit $\delta$ of the adjunction is not natural on the whole of $\cC$ (\Cref{rem:delta_nonnatural}), the strength $\sigma \colon X\otimes PY\to P(X\otimes Y)$ is natural with respect to general morphisms only in the second argument, and natural with respect to deterministic morphisms in the first argument.
	\end{remark}

	\begin{corollary}
		\label{cor:markov_to_kleisli1}
		Let $\cC$ be a representable Markov category. 
		Then the monad $(P, P \samp, \delta)$ on $\cC_{\det}$ arising from the underlying adjunction is symmetric monoidal and affine, thus inducing an isomorphism of Markov categories $\cC \cong \Kl(P)$.
	\end{corollary}
	\begin{proof}
		We have already noted that there is a canonical isomorphism of categories $\cC \cong \Kl(P)$.
		It is also an isomorphism of \emph{monoidal} categories because the defining adjunction $\cC_{\det}(A,PX) \cong \cC(A,X)$ is monoidal. 
		Finally, to see that the copy maps are preserved, it is enough to note that on both sides, they are given by the diagonals $Y \to Y \times Y$ in the cartesian monoidal category $\cC_\det$.
	\end{proof}

	\begin{lemma}
		\label{pullback_holds}
		Let $\cC$ be a representable Markov category with distribution functor $P$. 
		Then $P$ satisfies the pullback condition of \Cref{det_is_nice} on $\cC_\det$.
	\end{lemma}
	\begin{proof}
		We need to show that for every $A,X \in \cC$ and any diagram in $\cC_\det$ of the form 
		\begin{equation}\label{eq:pullback}
			\begin{tikzcd}
				A \ar[bend left]{drr}{f_1} \ar[bend right,swap]{ddr}{f_2} \ar[dashed]{dr}{g}							\\
				& X \ar{r}{\delta} \ar[swap]{d}{(\delta,\delta)} \ar[dr,phantom," "{pullback},very near start]	& PX \ar{d}{P(\id,\id)}		\\
					& PX \otimes PX \ar{r}{\nabla}									& P(X \otimes X)
			\end{tikzcd}
		\end{equation}
		without the dashed arrow, there is a unique dashed arrow such that the diagram commutes. Note that the diagonal in $\cC_\det$ is given by the copy map in $\cC$, so that the two vertical morphisms in the diagram are
		\begin{align}
			(\delta,\delta) &= (\delta \otimes \delta) \circ \cop_X,    &   P(\id,\id) &= P(\cop_X).
		\end{align}
		Since $\delta$ is a monomorphism by $\samp \circ \delta = \id$, the uniqueness is automatic and it is enough to find some $g$ which makes the diagram commute. 
		
		To this end, note first that composing the whole diagram with the two marginalization maps $P(X \otimes X) \to PX$ shows that $f_2 = (f_1, f_1)$, again as a pairing with respect to the universal property of $PX \otimes PX$ as a product in $\cC_\det$.
		
		We now show that $g \coloneqq \samp \circ f_1$ does the job. 
		To this end, it is enough to prove that $g$ is deterministic, because then we have
		\begin{equation}
			\delta \circ g = Pg \circ \delta
		\end{equation}
		resulting in commutativity of the upper triangle by
		\begin{equation}
			Pg \circ \delta = P \samp \circ P f_1 \circ \delta = f_1 \circ \samp \circ \delta = f_1
		\end{equation}
		Commutativity of the lower left triangle then also follows, thanks to $f_2 = (f_1, f_1)$.

		The claim that $g$ is deterministic amounts to the commutativity of the outermost rectangle in the diagram
		\[
			\begin{tikzcd}
				A \ar{rr}{f_1} \ar{ddrr}{f_2} \ar[swap]{dd}{\cop}	&& PX \ar{rr}{\samp} \ar[swap]{dr}{P(\cop)}	&	& X \ar{dd}{\cop} 	\\
										&&							& P(X \otimes X) \ar{dr}{\samp}	\\
			A \otimes A \ar[swap]{rr}{f_1 \otimes f_1}		&& PX \otimes PX \ar{ur}{\nabla} \ar[swap]{rr}{\samp\, \otimes\, \samp}	&	& X \otimes X 
			\end{tikzcd}
		\]
		Here, the lower left triangle commutes by $f_2 = (f_1, f_1)$, the oddly shaped square by assumption, the upper right square by naturality of $\samp$ and the lower right triangle by \Cref{prop:monoidal_adjunction}.
		Therefore, choosing $g$ to be $\samp \circ f_1$ makes the diagram~\eqref{eq:pullback} commute.
	\end{proof}

	We can summarize the previous results as follows.\pagebreak[2]
	
	\begin{theorem}
		\label{markov_as_kleisli}
		For a Markov category $\cC$, the following are equivalent:
		\begin{enumerate}
			\item \label{it:rep} $\cC$ is representable.
			\item \label{it:monad} There is an affine symmetric monoidal monad $P$ on $\cC_\det$ such that:
				\begin{itemize}
					\item The diagram~\eqref{delta_pullback} is a pullback for every $X$.
					\item The identity functor on $\cC_\det$ extends to an isomorphism of Markov categories $\cC \cong \Kl(P)$.
				\end{itemize}
		\end{enumerate}
	\end{theorem}
	In particular, since representability is a property rather than extra structure, the monoidal monad $P$ in the second condition is unique (up to unique isomorphism).
	
	This result is similar, but unrelated, to \cite[Theorem~4.7]{callbyvalue}, where a correspondence is drawn between Freyd categories and Kleisli categories of \emph{strong} (not necessarily commutative) monads. 
	
	\begin{proof}
		If $\cC$ is representable, \Cref{lem:std_adjunction} gives us the desired monad $P$, and the isomorphism of Markov categories is the one of \Cref{cor:markov_to_kleisli1}. Finally,	\Cref{pullback_holds} states exactly that this monad satisfies the pullback condition. 
		
		For the converse, we only need to show that the inclusion functor $\cC_\det \hookrightarrow \cC$ has a right adjoint. This holds by the assumed isomorphism $\cC \cong \Kl(P)$ together with the Kleisli adjunction, which gives us natural bijections
		\[
			\cC_\det(A,PX) \cong \Kl(P)(A,X) \cong \cC(A,X).
		\]
		Note that the pullback condition is not needed in this argument.
	\end{proof}

	\begin{example}\label{ex:param_representable}
		If $\cC$ is a representable Markov category, then every parametric Markov category $\cC_\B$ as introduced in \Cref{sec:parametric} is representable too.
		One can use the same distribution objects and take the sampling map $\param{\samp_X}$ in $\cC_\B$ to be represented by $\discard_\B \otimes \samp_X$, resulting in the desired bijection
		\[
			\cC_{\B,\det}(A,P_\B X) = \cC_\det(\B \otimes A, PX) \cong \cC(\B \otimes A, X) = \cC_\B(A,X).
		\]
		Thus, the distribution functor $P_\B  \colon \cC_\B \to \cC_{\B,\det}$ acts the same on objects as the original \mbox{$P\colon \cC \to \cC_{\det}$} does. 
		The action on morphisms is then uniquely determined subject to making the bijection $\cC_{\B,\det}(A,P_\B X) \cong \cC_\B(A,X)$ natural. 
		Concretely, a morphism $\param{f \in \cC_\B(A,X)}$ represented by $f \colon \B \otimes A \to X$ gets mapped to the morphism $\param{P_\B f \in \cC_\B(P_\B A,P_\B X)}$ represented by the composite
		\begin{equation}\label{eq:P_param}
			\B \otimes PA \xrightarrow{\mkern10mu \sigma \mkern10mu}  P(\B \otimes A) \xrightarrow{\mkern7mu Pf \mkern7mu}  PX,
		\end{equation}
		where $\sigma \colon \B \otimes PA \to P(\B \otimes A)$ denotes the strength of the monad $P$ as introduced in \Cref{strength}. 
		
		In order to see that this is how $P_\B$ must act on morphism in $\cC_\B$, it is enough to show that this prescription indeed makes the bijection
		\[
			\cC_{\B,\det}(A,P_\B X) \cong \cC_\B(A,X)
		\]
		natural in $X$.
		That is, we need to check that for each morphism $\param{k \colon X \to Y}$ represented by $k \colon \B \otimes X\to Y$, the following diagram
		\begin{equation*}
			\begin{tikzcd}[sep=large,color=parametrized]
				\cC_{\B,\det}(A,P_\B X) \ar{d}{P_\B k \circ \phsm} \ar{r}{\samp \circ \phsm}[swap]{\cong} & \cC_{\B}(A,X) \ar{d}{k \circ \phsm} \\
				\cC_{\B,\det}(A, P_\B Y) \ar{r}{\samp \circ \phsm}[swap]{\cong} & \cC_{\B}(A,Y)
			\end{tikzcd}
		\end{equation*}
		commutes. 
		Starting with a deterministic $f \colon \B \otimes A\to PX$ in the top-left corner, commutativity of the diagram amounts to showing that the equation
		\begin{equation}
			\input{param_functor.tikz}
		\end{equation}
		holds in $\cC$, where we use \eqref{eq:P_param} in order to express $P_\B k$ in terms of $Pk$ and $\sigma$.
		This equation follows straightforwardly if we apply the naturality of $\samp$ and property~\eqref{eq:strength_nat}.
	\end{example}
		
\subsection{Almost-Surely-Compatible Representability}
\label{sec:ascr}
		
	In a representable Markov category, it is not a priori clear whether the defining equation $\cC_\det(A,PX) \cong \cC(A,X)$ respects almost sure equality, in the following sense.
	An almost sure equality of two deterministic morphisms $A \to PX$ (with respect to some morphism $p \colon \Theta \to A$) implies the corresponding almost sure equality of the resulting morphisms $A \to X$, since the latter are obtained simply by composition with the sampling morphisms $PX \to X$. 
	However, the other direction is not clear: 
	If two morphisms $A \to X$ are almost surely equal, does this means that also their deterministic counterparts $A \to PX$ must be almost surely equal?

	Indeed, in \Cref{ex:not_as_rep} we provide a representable Markov category in which this converse implication fails to hold. 
	But since such a converse is relevant to our upcoming applications of representable Markov categories, we now investigate representable Markov categories in which it does hold.

	\begin{definition}
		\label{def:adjunction_as_compatible}
		A Markov category is \emph{\as{}-compatibly representable} if it is representable and for any morphism $p \colon \Theta \to A$, the defining natural bijection
		\[
			\cC_{\det}(A,PX) \cong \cC(A,X)
		\]
		respects almost sure equality.
		That is, for all $f,g \colon A \to X$, we have
		\begin{equation}\label{eq:adjunction_as_compatible}
			f^{\sharp} \ase{p} g^{\sharp}  \qquad \iff \qquad  f \ase{p} g.
		\end{equation}
	\end{definition}

	As we already noted, the implication from left to right is automatic because of $f = \samp \, f^\sharp$.

	Many representable Markov categories are actually \as{}-compatibly representable, including $\BorelStoch$ as the following example shows.

	\begin{example}
		\label{ex:borelstoch_ascompatible}
		For any two $f,g \colon A \to X$ in $\BorelStoch$, we have $f \ase{p} g$ if and only if for all $S \subseteq \Sigma_A$ and $T \subseteq \Sigma_X$,
		\begin{equation}
			\int_S f(T|a) \, p(da|\theta) = \int_S g(T|a) \, p(da|\theta),
		\end{equation}
		or equivalently if and only if the two functions $f(T|\ph), g(T|\ph) \colon A \to \R$ are \as{$p(\ph|\theta)$} equal for every $\theta \in \Theta$ and every $T \in \Sigma_X$~\cite[Example~13.3]{fritz2019synthetic}.
		What we need to prove is that this holds uniformly in $T$, i.e.\ that the \emph{measures} $f(\ph|a)$ and $g(\ph|a)$ are likewise $p(\ph|\theta)$-almost surely equal with respect to $a \in A$. 
		Since $\Sigma_X$ is countably generated, say by a sequence of measurable sets $(T_n)_{n \in \N}$, it is enough to show that $f(T_n|a) = g(T_n|a)$ holds for all $n$ with unit probability in $a$. 
		But this is indeed the case by assumption, since a countable intersection of sets of full measure again has full measure.
		Therefore, $\BorelStoch$ is \as{}-compatibly representable. 
	\end{example}

	A property equivalent to \as{}-compatible representability, which is useful in manipulations of string diagrams, turns out to be the following.
	\begin{definition}
		\label{def:scp}
		A representable Markov category is said to satisfy the \emph{sampling cancellation property} if, for any three morphisms $f,g \colon X \otimes A \to Y$ and $h \colon A \to X$, the following implication holds:
		\begin{equation*}\label{eq:scp}
			\input{scp.tikz}
		\end{equation*}
	\end{definition}
	The name of this condition is explained by the equation $f = \samp \, f^\sharp$, so that the implication amounts to the possibility to cancel the sampling map in diagram equations of the above form.

	\begin{proposition}
		\label{prop:scp_vs_ascomp}
		A representable Markov category $\cC$ satisfies the sampling cancellation property if and only if it is \as{}-compatibly representable.
	\end{proposition}
	\begin{proof}
		If in \Cref{def:scp} we use $f,g \colon X \otimes A \to Y$ of the form $f \otimes \discard_A$ and $g \otimes \discard_A$ for arbitrary $f,g \colon X \to Y$ and $h = p$, then we recover the non-trivial 
			direction of (\ref{eq:adjunction_as_compatible}).
			
		Conversely, if in \Cref{def:adjunction_as_compatible} we use $p \colon A \to X \otimes A$ given by 
		\begin{equation}\label{eq:inputcopy2}
			\input{inputcopy2.tikz}
		\end{equation}
		 for a given $h \colon A \to X$, then the right-to-left implication of (\ref{eq:adjunction_as_compatible}) implies the sampling cancellation property.
	\end{proof}

	In \Cref{ex:param_representable}, we saw that if a Markov category $\cC$ is representable, then so is every parametric Markov category $\cC_\B$. 
	As the following lemma shows, the same can be said about \as{}-compatible representability.
	
	\begin{lemma}\label{lem:parametric_ascr}
		Let $\cC$ be an \as{}-compatibly representable Markov category.
		For every $\B \in \cC$, the parametric Markov category $\cC_\B$ as introduced in \Cref{sec:parametric} is likewise \as{}-compatibly representable.
	\end{lemma}
	\begin{proof}
		First of all, notice that if $\param{f \in \cC_\B(A,X)}$ is a morphism in $\cC_\B$ represented by \mbox{$f \colon \B \otimes A \to X$} in $\cC$, then its adjunct $\param{f^\sharp \in \cC_\B(A,P_\B X)}$ is represented by \mbox{$f^\sharp \colon \B \otimes A \to PX$}, the adjunct of $f$ in $\cC$.
		This follows because $P_\B X$ is a distribution object of $\cC_\B$ with respect to the same sampling map as $PX$ is in $\cC$.
		
		Therefore, the non-trivial part of checking \as{}-compatible representability of $\cC_\B$, i.e.\ the right-to-left implication of \eqref{eq:adjunction_as_compatible}, boils down to the following implication
		\begin{equation}\label{eq:parametric_ascr}
			\input{parametric_ascr.tikz}
		\end{equation}
		in $\cC$. 
		This holds because $\cC$ satisfies the sampling cancellation property by \Cref{prop:scp_vs_ascomp}.
		Consequently, $\cC_\B$ is \as{}-compatibly representable.
	\end{proof}

	\begin{example}\label{ex:not_as_rep}
		We now give an example of a Markov category which is representable but not \as{}-compatibly representable. 
		Continuing on from \Cref{entire}, the Kleisli category $\Kl(D_R)$ of the distribution monad $D_R$ for an entire commutative semiring $R$ is representable with $\Kl(D_R)_\det = \cat{Set}$. 
		As we elaborate below, there is an $R$ such that $\Kl(D_R)$ is not \as{}-compatibly representable.

		Concretely, let $R$ be the semiring
		\begin{equation}
			R \coloneqq \{0,\varepsilon,1\}
		\end{equation}
		with addition and multiplication given by the following nontrivial cases:
		\begin{gather}\begin{aligned}
			1 + 1 & = 1,	& 1 + \varepsilon & = 1,	& \varepsilon + \varepsilon &= \varepsilon,	\\[4pt]
					&& \varepsilon^2 &= \varepsilon. &&
		\end{aligned}\end{gather}
		Intuitively, we can think of assigning value $1 \in R$ to outcomes that are possible and happen with nonzero probability; assigning $\varepsilon$ to outcomes that may be considered, but whose probability is $0$ or negligibly small; and assigning $0 \in R$ to outcomes that could never happen. Then the above arithmetical rules acquire straightforward interpretations.

		Since this commutative semiring $R$ has idempotent addition and multiplication, it can also be understood as a distributive lattice with $\max$ as addition and $\min$ as multiplication. In this picture, $R$ is simply the three-element totally ordered set with $0 < \varepsilon < 1$. 
		It is also clear that $R$ is entire.

		In the representable Markov category $\Kl(D_R)$, we then consider $A = X = \{a,b\}$ together with $f,g \colon A \to A$ represented by the morphisms \mbox{$f^\sharp,g^\sharp \colon A \to PA$} given by
		\begin{align*}
			f^\sharp(a) & \coloneqq \delta_a,	& f^\sharp(b) & \coloneqq \varepsilon \delta_a + \phantom{\varepsilon} \delta_b,	\\
			g^\sharp(a) & \coloneqq \delta_a,	& g^\sharp(b) & \coloneqq \phantom{\varepsilon} \delta_a + \varepsilon	\delta_b.
		\end{align*}
		We also consider $p \colon I \to A$ represented by $p^\sharp \coloneqq \delta_a + \varepsilon \delta_b$. Then we have $f^\sharp \not\ase{p} g^\sharp$, since
		\begin{equation}
			\delta_{(a,\delta_a)} + \varepsilon \delta_{(b,\varepsilon \delta_a + \delta_b)} \neq 
				\delta_{(a,\delta_a)} + \varepsilon \delta_{(b,\delta_a + \varepsilon \delta_b)}.
		\end{equation}
		However, applying sampling to the second output reduces this to
		\begin{equation}
			\delta_{(a,a)} + \varepsilon^2 \delta_{(b,a)} + \varepsilon \delta_{(b,b)} =
				\delta_{(a,a)} + \varepsilon \delta_{(b,a)}  + \varepsilon^2 \delta_{(b,b)},
		\end{equation}
		where the equation holds by $\varepsilon^2 = \varepsilon$. 
		Therefore, we have $f \ase{p} g$ and $\Kl(D_R)$ is not \as{}-compatibly representable.
	\end{example}

\section{Second-Order Stochastic Dominance}\label{sec:2ndorder}

	In this and the following sections, we state and prove generalizations of existing concepts in probability theory and statistics to representable Markov categories.
	The first one for which we do this is \emph{second-order stochastic dominance}. 
	Traditionally, this is a partial order of probability distributions on $\R$ or $\R^n$ that expresses whether a given distribution is more ``spread out'' than another.

	In the abstract setting of representable Markov categories, second-order dominance makes sense for all algebras $A$ of the monad $P$ on $\cC_\det$. 
	It is concerned with comparing two distributions on $A$, represented now by morphisms $p,q \colon I \to A$, and induces a preorder relation on $\cC(I,A)$. 
	More generally, it is a preorder relation on \emph{every} hom-set $\cC(\Theta,A)$, defined in terms of the $P$-algebra structure on $A$.

	One way to define second-order dominance in terms of the existence of a so-called \emph{dilation} $A \to A$ that takes $q$ to $p$.
	Intuitively, a dilation is a specific kind of map which increases uncertainty without affecting the expected value of distributions on~$A$.

	\begin{definition}\label{def:dilationgeneral}
		Let $\cC$ be a representable Markov category and let $e \colon PA \to A$ be an algebra of the monad $(P, \delta, P\samp)$ on $\cC_{\det}$. 
		Then, given a morphism $f \colon \Theta \to A$ in $\cC$, a morphism $t \colon \Theta \tensor A \to A$ is an \emph{$f$-dilation} (with respect to the $P$-algebra structure~$e$) if it satisfies
		\begin{equation}\label{eq:dilationgeneral}
			\input{dilationgeneral.tikz}
		\end{equation}
	\end{definition}

	Thus, in case $\Theta$ is the unit object, an $f$-dilation $t$ has to satisfy 
	\begin{equation}
		e \circ t^\sharp \ase{f} \id.
	\end{equation}
	In the probability theory context, such dilations are also called \emph{mean-preserving maps}, as the following example elucidates.

	\begin{example}
		\label{dilation_Rn}
		Given the choice of $\cC = \BorelStoch$ and a compact set $A \subseteq \R$, the canonical algebra map $e \colon P A \to A$ is the assignment of expectation values.
		Relative to it, an $f$-dilation is any Markov kernel $t \colon A \to PA$ that takes $f$-almost every point $a \in A$ to a distribution $t^\sharp(a)$ whose expectation value is $a$ itself. 
		 Similarly, for a compact $A \subseteq \R^n$, or for a closed and bounded subset of a separable Banach space, the natural choice of $e \colon PA \to A$ is one that maps each distribution to its barycenter.

		For general $\Theta$, \cref{eq:dilationgeneral} in $\cat{BorelStoch}$ becomes
		\begin{equation}
			e \bigl( t(\ph|\theta,a) \bigr) = a
		\end{equation}
		for $f(\ph|\theta)$-almost all $a \in A$ and all $\theta \in \Theta$. This again matches the intuition that the distribution $t(\ph|\theta,a)$ dilates the point $a$, suitably almost surely, but now for every $\theta \in \Theta$ separately.
	\end{example}	

	\begin{remark}
		\label{no_param}
		Throughout this section, we could restrict to the case $\Theta = I$ only, and then instantiate our definitions and results on the parametric Markov category $\cC_\Theta$ in order to recover the additional parameter dependence on any $\Theta \in \cC$. 
		We have decided against doing so for the benefit of making the greater generality of our formalism more explicitly apparent.
	\end{remark}
	
	When working with second-order dominance, it is often helpful to assume \as{}-compatible representability rather than mere representability, which we do from now on.

	\begin{lemma}\label{lem:dilationfree}
		Let $\cC$ be an \as{}-compatibly representable Markov category.
		Consider the free algebra $PX$ of the monad $(P, P\samp, \delta)$ on $\cC_{\det}$, with algebra map
		\[
			P \samp \colon PPX \to PX,
		\]
		and a morphism $f \colon \Theta \to PX$ in $\cC$.
		Then $t \colon \Theta \tensor PX \to PX$ is an $f$-dilation if and only if it satisfies
		\begin{equation}\label{eq:dilationfree}
			\input{dilationfree.tikz}
		\end{equation}
	\end{lemma}
	\begin{proof}
		\Cref{eq:dilationfree} follows from \eqref{eq:dilationgeneral} for $e = P\samp$, namely by applying $\id_{PX} \otimes \samp_X$ to it and using the commutativity of diagram~\eqref{samp_associative} together with $\samp \circ t^\sharp = t$.
		
		Conversely, since the sampling cancellation property follows from \as{}-compatible representability by \Cref{prop:scp_vs_ascomp}, we can follow the same reasoning backwards\footnote{Note that we cannot merely remove the sampling maps since $t$ is generally not deterministic.}
		and conclude the equation 
		\begin{equation}\label{eq:dilationfree2}
			\input{dilationfree2.tikz}
		\end{equation}
		as was to be shown.
	\end{proof}

	As the following proposition shows, the existence of dilations that convert between morphisms can be related to the existence of partial evaluations \cite{fritz2020monads}.
	\begin{proposition}\label{prop:2ndorder}
		Let $\cC$ be an \as{}-compatibly representable Markov category with distribution functor $P$ and for which conditionals exist. Then for any $P$-algebra $e \colon PA \to A$ in $\cC_\det$ and any morphisms $p,q  \colon \Theta \to A$ in $\cC$, the following two conditions are equivalent:
		\begin{enumerate}[label=(\roman*)]
			\item\label{it:pev} There exists a morphism $r \colon \Theta \to PA$ such that the following diagram commutes:
				\begin{equation}\label{eq:pevcond}
					\begin{tikzcd}
						\Theta \ar[bend left]{drr}{p} \ar[swap,bend right]{ddr}{q} \ar{dr}[description]{r}				\\
													& PA \ar[swap]{r}{\samp} \ar{d}{e}	& A	\\
													& A
					\end{tikzcd}
				\end{equation}
				
			\item \label{it:dilation} There exists a $q$-dilation $t \colon \Theta \tensor A \to A$ which converts $q$ to $p$ in the following sense:
				\begin{equation}\label{eq:measure_preserving}
					\input{dilation7.tikz}
				\end{equation}
		\end{enumerate}
	\end{proposition}
	\begin{proof}
		We split the proof into two parts each consisting of one of the implications.
		\begin{enumerate}
			\item[\ref{it:pev} $\Rightarrow$ \ref{it:dilation}:]
				Given an $r \colon \Theta \to PA$ that makes the diagram (\ref{eq:pevcond}) commute, we construct $s \colon A \otimes \Theta \to A$ as a conditional of the morphism
				\begin{equation}\label{eq:2ndorder}
					\input{2ndorder_v3.tikz}
				\end{equation}
				with respect to the left output, so that 
				the defining equation of conditionals
				\begin{equation}\label{eq:2ndorder2}
					\input{2ndorder2_v3.tikz}
				\end{equation}
				holds.
				Upon defining $t \coloneqq \samp_A \circ s$, the property (\ref{eq:measure_preserving}) then follows immediately by applying $\discard_A \tensor \id_{PA}$ to equation (\ref{eq:2ndorder2}).

				It remains to be shown that $t$ is a $q$-dilation. 
				Applying $\id_A \otimes e$ to equation (\ref{eq:2ndorder2}), we get
				\begin{equation}\label{eq:2ndorder3}
					\input{2ndorder3_v3.tikz}
				\end{equation}
				where the second step uses the assumption that $e$ is deterministic. Applying the sampling cancellation property, we obtain
				\begin{equation}\label{eq:2ndorder6}
					\input{2ndorder6.tikz}
				\end{equation}
				Now the fact that $t^{\sharp} = P\samp \circ s^{\sharp}$ satisfies equation (\ref{eq:dilationgeneral}) follows by applying $\id_A \tensor e$ to equation (\ref{eq:2ndorder6}) and using the fact that 
					$e$ is a $P$-algebra:
				\begin{equation}\label{eq:2ndorder7}
					\input{2ndorder7.tikz}
				\end{equation}
				Thus, we have shown the implication \ref{it:pev} $\Rightarrow$ \ref{it:dilation}.
		
			\item[\ref{it:dilation} $\Rightarrow$ \ref{it:pev}:]
				Given a morphism $t \colon \Theta \tensor A \to A$ satisfying (\ref{eq:measure_preserving}) and (\ref{eq:dilationgeneral}), we define $r$ via 
				\begin{equation}\label{eq:pev}
					\input{pev_v2.tikz}
				\end{equation}
				Applying $\samp_A$ to this equation yields
				\begin{equation}\label{eq:pev2}
					\input{pev2.tikz}
				\end{equation}
				whence $\samp_A r = p$ follows by virtue of the fact that $t$ converts $q$ to $p$.
				
				On the other hand, applying $e$ to equation (\ref{eq:pev}) gives $e \circ r = q$ by the assumption that $t$ is a $q$-dilation with respect to $e$. 				\qedhere
		\end{enumerate}
	\end{proof}
	\begin{definition}
		\label{2ndorder_stoch_dom}
		If the equivalent conditions of \Cref{prop:2ndorder} hold, then we say that $q$ \emph{second-order dominates} $p$ with respect to $e$, and denote this relation as
		\begin{equation}
			p \sqsubseteq q.
		\end{equation}
	\end{definition}

	By \Cref{dilation_Rn}, this recovers the usual notion of second-order stochastic dominance for probability measures on any compact $A \subseteq \R^n$ if we take $\cC$ to be $\BorelStoch$ and $\Theta = I$.

\section{Comparison of Statistical Experiments}\label{sec:blackwell}

\subsection{Informativeness of Statistical Experiments}\label{sec:informativeness}

	Dilations also appear in the context of comparison of statistical experiments, as introduced by Blackwell \cite{blackwell1951comparison}.
	In this case, statistical experiments are modeled by families of measures indexed by a parameter set $\Theta$ that labels the distinct hypotheses one aims to discern by performing the experiment.
	Here, we represent statistical experiments abstractly as morphisms $\Theta \to X$ in a Markov category.
	
	Commonly, the question of comparing two experiments $f,g$ in terms of whether $f$ is \emph{more informative} about $\Theta$ than $g$ is, amounts to the existence of a map $c$ satisfying $c \, f = g$.
	We can interpret $c$ as a post-processing of the data generated by experiment $f$ in a way that produces the data of experiment $g$.
	If such a $c$ exists, we refer to it as \emph{garbling map} and say that $f$ is more informative than $g$, denoted by $f \succeq g$.
	A slightly more general version of this notion was introduced by Golubtsov in \cite{golubtsov2002kleisli}.
	
	\begin{example}\label{ex:rod}
		For a concrete example, consider a company that manufactures metal rods.
		After production, a selection of rods is tested for structural integrity.
		The hypothesis space $\Theta =\{\safe,\faulty\}$ thus consists of two possibilities: 
		Either the selected rod is safe ($\safe$) or it is faulty ($\faulty$).
		
		There are two tests available.
		The threshold test $g$ checks whether the rod withstands hanging a weight from its midpoint, while the oscillation test $f$ checks whether the rod withstands transverse oscillations of a certain type.		
		Each test has two possible outcomes:
		Either the rod passes ($\pass$) or fails ($\fail$) the test, so that we have $X = Y = \{\pass,\fail\}$.
		According to our model of metal rods and their defects (that we assume to be correct, up to the experimental precision) the outcomes of the two tests are generally distributed as follows:
		\begin{align*}
			g (\pass | \safe) &= 0.72,  &  g (\pass | \faulty) &= 0.45,  &&&  f (\pass | \safe) &= 0.96,  &  f (\pass | \faulty) &= 0.6, \\
			g (\fail | \safe) &= 0.28,  &  g (\fail | \faulty) &= 0.55,  &&&  f (\fail | \safe) &= 0.04,  &  f (\fail | \faulty) &= 0.4. 
		\end{align*}
		However, suppose that the test results cannot be trusted when both are executed on the same rod, for example because either test may afflict structural damage to the rod even upon passing the test.
		Then only one of $g$ or $f$ can be performed, and we are interested to know whether doing one of the experiments provides us with strictly more information about whether a given rod is safe.
		To that end, we can use the ordering $\succeq$.
		In particular, we have $f \succeq g$, because the stochastic map
		\begin{align*}
			c(\pass | \pass) &= 0.75  &  c(\pass | \fail) &= 0 \\
			c(\fail | \pass) &= 0.25  &  c(\fail | \fail) &= 1
		\end{align*}
		achieves the conversion of the oscillation test to the threshold test by garbling.
	\end{example}
	
	The informativeness ordering defined via garbling maps can be equivalently expressed in terms of loss functions or Bayesian utilities in decision theory.
	For an account of the origins of this equivalence, see \cite[Section 3]{le1996comparison} for example.
	Recently, it has also been expressed in categorical terms by de Oliveira in \cite{de2018blackwell}.

	In what follows, we make use of the following relaxed informativeness ordering relative to a particular prior $m \colon I \to \Theta$.
	\begin{definition}\label{def:informativeness}
		Given morphisms $m \colon I \to \Theta$, $f \colon \Theta \to X$ and $g \colon \Theta \to Y$ in any Markov category $\cC$, we say that $f$ is \emph{\as{$m$} more informative} than $g$, denoted $f \asinf{m} g$, if there exists 
			a morphism $c \colon X \to Y$ such that we have $c\, f \ase{m} g$,
		\begin{equation}\label{eq:informativeness_m-as}
			\input{informativeness_m-as.tikz}
		\end{equation}
	\end{definition}

	Note that the informativeness orderings are actually \emph{preorders} on the collection of morphisms out of $\Theta$, but we use the term ``ordering'' as synonymous with ``preordering''.
	
	\begin{remark}
		The informativeness ordering and its approximate versions that we leave out of our analysis here are particularly interesting in situations when the two experiments $f$ and $g$ \emph{cannot} be implemented jointly or when it is costly to do so.
		They tell one which of the two mutually exclusive choices for an experiment is a better choice as far as learning about the hypothesis $\Theta$ is concerned.
		
		On the other hand, if one \emph{could} execute both $f$ and $g$ simultaneously, a more relevant question may be that of whether $f$ (or $g$) is more informative than their joint implementation, which would be an experiment $h \colon \Theta \to X \otimes Y$ with marginals $f$ and $g$.
		In practice, $X$ and $Y$ need not be conditionally independent given $\Theta$.
		However, it is customary to assume so, in accordance with the assumption that $\Theta$ constitutes a complete set of relevant parameters.
		The joint experiment can then be expressed as
		\begin{equation}\label{eq:joint}
			\input{joint.tikz}
		\end{equation}
		It is not hard to see that with this kind of joint implementation, the relation between $f$ and $g$ induced by $f \asinf{m} h$ is \emph{different} from the basic informativeness ordering $f \asinf{m} g$.
		Nevertheless, as \Cref{thm:informativeness_vs_sufficiency} below shows, the latter \emph{is} closely connected to the question of whether $X$ (or $Y$) constitutes a \emph{sufficient statistic} for $X \times Y$, but the joint implementation of $f$ and $g$ is not in general the one given by \cref{eq:joint} (in particular, see the left equation of~\eqref{eq:recovery_as_garbling}).
	
		The considerations in the paragraph above do not necessarily take into account that there might be a cost associated with performing the experiments, which would make the comparison of the information contained in two samples (one from $X$ and one from $Y$) with that of one sample (from $X$, say) less meaningful.\footnote{For example, consider having to make a choice between one of two large-scale medical trials, in a situation where conducting both would not be feasible.}
		Instead, if one compares the same number of samples, either one from $X$ or one from $Y$, then the relevant relation is the basic informativeness ordering given by $f \asinf{m} g$.
	
		A slightly different context in which the informativeness ordering is relevant is the theory of communication.
		There, we commonly interpret $f$ and $g$ as distinct encodings of $\Theta$ in $X$ and $Y$ respectively.
		In general, $f$ and $g$ lose some of the information contained in $\Theta$.
		The preorder $\succeq$ then tells us which of the two encodings unambiguously retains more of this information, and whether such a comparison can be made at all.
	\end{remark}
	
	A measure-theoretic version the following theorem appears in \cite[Theorem 7.2.16]{torgersen1991comparison}, but its roots can be traced back to \cite{bahadur1955characterization}.
	\begin{theorem}\label{thm:informativeness_vs_sufficiency}
		Let $m \colon I \to \Theta$, $f \colon \Theta \to X$ and $g \colon \Theta \to Y$ be morphisms in a Markov category $\cC$.
		If $\cC$ has conditionals, then the following are equivalent:
		\begin{enumerate}[label=(\roman*)]
			\item \label{it:informativeness} $f$ is \as{$m$} more informative than $g$.
			\item \label{it:sufficiency} There exists a morphism $h \colon \Theta \to X \otimes Y$ with marginals $m$-almost surely given by $f$ and $g$ respectively, such that the deterministic morphism $\id_X \otimes \discard_Y$ is a sufficient statistic\footnote{See \cite[Definition 14.3]{fritz2019synthetic} for a synthetic definition of sufficiency in Markov categories.} for $h$.
			\item \label{it:cond_indepenedence} There exists a morphism $\mu \colon I \to \Theta \otimes X \otimes Y$ satisfying
				\begin{align*}
					\input{informativeness_vs_sufficiency.tikz}
				\end{align*}
		\end{enumerate}
	\end{theorem}
	The reading of condition $(b)$ is that there are some morphisms that can take place of the empty boxes so that the equation holds.
	In other words, it states that $\mu$ displays the conditional independence \cite[Definition 12.1]{fritz2019synthetic} of $\Theta$ and $Y$ given $X$.
	Also note that conditions $(c)$ and $(d)$ are independent of the choice of conditionals by the \as{}-uniqueness of conditionals.
	\begin{proof}
		We split the proof into three implications.
		\begin{enumerate}
			\item[\ref{it:informativeness} $\Rightarrow$ \ref{it:sufficiency}:]
				The condition that $\id_X \otimes \discard_Y$ is a sufficient statistic for $h$ \cite[Definition 14.3]{fritz2019synthetic} translates to the existence of a morphism $\alpha \colon X \to X \otimes Y$ such that
				\begin{equation}\label{eq:sufficient_statistic}
					\input{sufficient_statistic.tikz}
				\end{equation}
				holds.
				Assuming condition \ref{it:informativeness} is true, we can use the garbling map $c$ satisfying \cref{eq:informativeness_m-as} to define the requisite $h$ and $\alpha$ via
				\begin{equation}\label{eq:recovery_as_garbling}
					\input{recovery_as_garbling.tikz}
				\end{equation}
				from which \cref{eq:sufficient_statistic} follows.
				Since the marginals of $h$ so defined are $f$ and $g$ respectively (the latter $m$-almost surely), we obtain the desired \mbox{implication \ref{it:informativeness} $\implies$ \ref{it:sufficiency}}.				
			
			\item[\ref{it:sufficiency} $\Rightarrow$ \ref{it:cond_indepenedence}:]
				We can define a morphism $\mu$ in terms of $h$ and $m$ as follows:
				\begin{equation}\label{eq:Markov_triple_from_joint}
					\input{Markov_triple_from_joint.tikz}
				\end{equation}
				Conditions $(a)$, $(c)$, and $(d)$ are then immediate.
				In order to show that condition $(b)$ holds as well, we can use \cref{eq:sufficient_statistic} with the middle output marginalized to get
				\begin{equation}\label{eq:cond_indepenedence_proof}
					\input{cond_indepenedence_proof.tikz}
				\end{equation}
				where the second equation follows from assuming condition~\ref{it:sufficiency}, in particular from the fact that the $\Theta \to X$ marginal of $h$ is \as{$m$} equal to $f$.
				In order to obtain condition $(b)$, we can then apply the definition of a Bayesian inverse of $f$ with respect to $m$ to get
				\begin{equation}\label{eq:cond_indepenedence_proof2}
					\input{cond_indepenedence_proof2.tikz}
				\end{equation}
				as required.
				Consequently, condition~\ref{it:sufficiency} indeed implies condition~\ref{it:cond_indepenedence}.
			
			\item[\ref{it:cond_indepenedence} $\Rightarrow$ \ref{it:informativeness}:]
				Let the decomposition of $\mu$ as given by condition $(b)$ be 
				\begin{equation}\label{eq:cond_indepenedence_decomp}
					\input{cond_indepenedence_decomp.tikz}
				\end{equation}
				It remains to show that the morphism $c \colon X \to Y$ from such a decomposition can act as the garbling that achieves the conversion of $f$ to $g$, $m$-almost surely.
				Taking the conditional of \cref{eq:cond_indepenedence_decomp} and using $k \, n = m$, which follows from condition $(a)$, yields
				\begin{equation}\label{eq:garbling_from_cond_indepenedence}
					\input{garbling_from_cond_indepenedence.tikz}
				\end{equation}
				where $k^\dagger$ denotes the Bayesian inverse of $k$ with respect to $n$.
				Marginalizing over $Y$ in \cref{eq:garbling_from_cond_indepenedence} and using condition $(c)$ then gives $k^\dagger \ase{m} f$.
				Finally, by marginalizing \cref{eq:garbling_from_cond_indepenedence} over $X$ and using condition $(d)$, we arrive at $g \ase{m} c \, f$ and the proof is thus complete.\qedhere
		\end{enumerate}
	\end{proof}

\subsection{The Classical Blackwell--Sherman--Stein Theorem}\label{sec:classical_BSS}
	
	Insofar as there is a host of equivalent ways to characterize the informativeness ordering \cite[Theorem~1]{le1996comparison}, we are interested in the extent to which one can generalize these results to the abstract setting of Markov categories
		and proved with synthetic methods.
	We already saw an instance of such a result in the form of \Cref{thm:informativeness_vs_sufficiency}, which relates informativeness to sufficient statistics and conditional independence.
	For the remainder of \Cref{sec:blackwell}, we focus on the \emph{Blackwell--Sherman--Stein Theorem} \cite[Theorem~6]{blackwell1951comparison}, also known as the \emph{dilation criterion}.
	It states that the informativeness order $\succeq$ for statistical experiments coincides with the second-order stochastic dominance order of the so-called \emph{standard measures} on $P \Theta$, the description of which we turn to now.
	
	An intuitive way to think of standard measures is to interpret a statistical experiment $f \colon \Theta \to X$ as a way of learning about the underlying hypothesis represented by the parameter set $\Theta$. 
	From this perspective, given a prior $m \in P \Theta$, one can use Bayesian updating to find the corresponding posterior on $P\Theta$, which of course depends on the value of $X$ observed.
	Thus, it can be represented by the measurable map $(f^{\dagger})^\sharp \colon X \to P\Theta$, which, recalling \Cref{not:sharp}, contains the same information as the Markov kernel $f^\dagger \colon X \to \Theta${\,\textemdash\,}the Bayesian inverse of $f$.
	The standard measure is then an element of $PP\Theta$ obtained as a mixture of these posteriors with respect to the chosen prior $m$, or more precisely, with respect to the distribution of $X$ one would expect, were the ``true'' value of the hypothesis $\Theta$ sampled from $m$.
	
	The morphism given by the composite
	\begin{equation}
		\begin{tikzcd}
			\Theta \ar{r}{f} & X \ar{r}{(f^\dagger)^\sharp} & P \Theta
		\end{tikzcd}
	\end{equation}
	is commonly referred to as the \emph{standard experiment} $\hat{f}$ of $f$.
	Intuitively, for a given ``true'' hypothesis as input, the standard experiment outputs the distribution over posteriors that results from conducting the experiment and applying Bayesian updating.
	Since the outcome of the experiment itself is random, we obtain a distribution over posteriors rather than a mere posterior.
	Therefore, the standard experiment is typically not deterministic and it clearly depends on the prior $m$.
	The standard measure is then nothing but $\hat{f}$ applied to $m$.

	\begin{example}\label{ex:rod_stndmeas}
		Traditionally, one uses a uniform prior $m$ to define standard measures for finite parameter sets, but any strictly positive probability measure would do.
		Continuing our \Cref{ex:rod} and choosing a uniform prior
		\begin{align*}
			m(\safe) &= 0.5  &  m(\faulty) &= 0.5
		\end{align*}
		gives the following posteriors (expressed as functions rather than Markov kernels):
		\begin{align*}
			\bigl(g^{\dagger}\bigr)^\sharp (\pass) &= \frac{0.72 \,\delta_{\safe} + 0.45 \,\delta_{\faulty}}{0.72 + 0.45} \approx 0.62 \,\delta_{\safe} + 0.38 \,\delta_{\faulty}  \\[4pt]
			\bigl(g^{\dagger}\bigr)^\sharp (\fail) &= \frac{0.28 \,\delta_{\safe} + 0.55 \,\delta_{\faulty}}{0.28 + 0.55} \approx 0.34 \,\delta_{\safe} + 0.66 \,\delta_{\faulty}  \\[4pt]
			\bigl(f^{\dagger}\bigr)^\sharp (\pass) &= \frac{0.96 \,\delta_{\safe} + 0.6 \,\delta_{\faulty}}{0.96 + 0.6} \approx 0.62 \,\delta_{\safe} + 0.38 \,\delta_{\faulty}  \\[4pt]
			\bigl(f^{\dagger}\bigr)^\sharp (\fail) &= \frac{0.04 \,\delta_{\safe} + 0.4 \,\delta_{\faulty}}{0.04 + 0.4} \approx 0.09 \,\delta_{\safe} + 0.91 \,\delta_{\faulty} .
		\end{align*}
		The expected distributions of test outcomes if exactly half of the rods were faulty are
		\begin{align*}
			g \circ m (\pass) &= 0.5 \, g (\pass | \safe) + 0.5 \, g (\pass | \faulty) = 0.585  \\
			g \circ m (\fail) &= 0.5 \, g (\fail | \safe) + 0.5 \, g (\fail | \faulty) = 0.415 \\
			f \circ m (\pass) &= 0.5 \, f (\pass | \safe) + 0.5 \, f (\pass | \faulty) = 0.78 \\
			f \circ m (\fail) &= 0.5 \, f (\fail | \safe) + 0.5 \, f (\fail | \faulty) = 0.22 \\
		\end{align*}
		Therefore the standard measures of $g$ and $f$, denoted by $\hat{g}_m$ and $\hat{f}_m$ respectively, are given by
		\begin{align*}
			\hat{g}_m &\approx 0.585 \, \delta_{0.62 \,\delta_{\safe} + 0.38 \,\delta_{\faulty}} + 0.415 \, \delta_{0.34 \,\delta_{\safe} + 0.66 \,\delta_{\faulty}}, \\
			\hat{f}_m &\approx 0.78 \, \delta_{0.62 \,\delta_{\safe} + 0.38 \,\delta_{\faulty}} + 0.22 \, \delta_{0.09 \,\delta_{\safe} + 0.91 \,\delta_{\faulty}}.
		\end{align*}
	\end{example}
	One may wonder why would one consider such a seemingly convoluted way to view a statistical experiment in terms of its standard measure.
	As mentioned at the beginning of \Cref{sec:classical_BSS}, one reason is that we can express comparison of statistical experiments in terms of the second-order dominance ordering among their standard measures.
	Explicitly, this is the classical version of the BSS Theorem.
	One of its strong points is that it reduces the comparison of experiments $\Theta \to X$ for \emph{arbitrary} $X$ to the comparison of measures on a \emph{fixed} sample space, namely $P \Theta$.
	\begin{theorem}[Blackwell--Sherman--Stein \cite{le1996comparison}]\label{thm:BSS}
		Let $\Theta$, $X$, and $Y$ be standard Borel spaces with $\Theta$ finite. 
		For any two Markov kernels $f \colon \Theta \to X$ and $g \colon \Theta \to Y$, the following are equivalent:
		\begin{enumerate}
			\item $f \succeq g$, i.e.\ $f$ is more informative about $\Theta$ than $g$ is.
			\item $\hat{f}_m \sqsubseteq \hat{g}_m$, i.e.\ the standard measure of $g$ second-order dominates the standard measure of $f$.
		\end{enumerate}
	\end{theorem}
	\begin{example}\label{ex:rod_dilation}
		In terms of our running example of metal rods, we saw that $f$ is more informative than $g$ in \Cref{ex:rod} and constructed the standard measures in \Cref{ex:rod_stndmeas}.
		\Cref{thm:BSS} says that there should be a $\hat{g}_m$-dilation $t \colon P \Theta \to P\Theta$ that maps $\hat{g}_m$ to $\hat{f}_m$, meaning that the posteriors according to the oscillation test $f$ are ``more spread out'' than those corresponding to the threshold test $g$.
		Concretely, a dilation that does the job can be chosen to be any measurable map $t^\sharp$ satisfying
		\begin{align*}
			0.62 \,\delta_{\safe} + 0.38 \,\delta_{\faulty} &\mapsto \delta_{0.62 \,\delta_{\safe} + 0.38 \,\delta_{\faulty}} \\[4pt]
			0.34 \,\delta_{\safe} + 0.66 \,\delta_{\faulty} &\mapsto 0.47 \, \delta_{0.62 \,\delta_{\safe} + 0.38 \,\delta_{\faulty}} + 0.53 \, \delta_{0.09 \,\delta_{\safe} + 0.91 \,\delta_{\faulty}}
		\end{align*}
		up to our convention of rounding to two decimal places.
		One can use \Cref{lem:dilationfree} to convince oneself that these relations indeed give a $\hat{g}_m$-dilation $t$, since we have
		\begin{align*}
			0.34 &\approx 0.47 * 0.62 + 0.53 * 0.09,  &  0.66 &\approx 0.47 * 0.38 + 0.53 * 0.91.
		\end{align*}
	\end{example}
	
\subsection{The Blackwell--Sherman--Stein Theorem in Markov Categories}\label{sec:BSS}

	Our goal is now to state and prove a version of \Cref{thm:BSS} for Markov categories.
	In order to arrive at such a synthetic generalization, we need a bit more than just a representable Markov category. 
	Indeed, the requirements are the same as in \Cref{prop:2ndorder}.

     \begin{assumption}
		\label{ass:ascr}
		Throughout the rest of \Cref{sec:blackwell}, let $\cC$ be an \as{}-compatibly representable Markov category with conditionals, as well as $f \colon \Theta \to X$ and $g \colon \Theta \to Y$ arbitrary morphisms in $\cC$ with the same domain.
	\end{assumption}

	We mention this assumption again in the statements of our results, but otherwise leave it implicit.
	
	The existence of conditionals is necessary in order to have an abstract notion of Bayesian inference, which is used to construct the basic elements of the BSS Theorem: standard experiments and standard measures.
	The representability is relevant again for the definition of the standard experiment and standard measure, which make reference to the distribution functor $P$.
	The \as{}-compatibility, in the sense of \Cref{def:adjunction_as_compatible}, is relevant for proving \as{} uniqueness of the standard experiment, and for the interpretation of the second condition of the upcoming \Cref{thm:dilation_criterion} as a second-order dominance relation via \Cref{lem:dilationfree}.
	
	Next, we present the definitions of standard experiment and standard measure in the language of Markov categories.
	\begin{definition}\label{def:stn_exp}
		Given morphisms $m \colon I \to\ \Theta$ and $f \colon \Theta \to X$ in a Markov category $\cC$ with conditionals, the \emph{standard experiment} $\hat{f} \colon \Theta \to P \Theta$ of $f$ is given by
		\begin{equation}\label{eq:stnd_exp}
			\input{stndexp.tikz}
		\end{equation}
		where $f^{\dagger} \colon X \to \Theta$ is a Bayesian inverse of $f$ with respect to $m$.
	\end{definition}
	Standard experiments are unique, $m$-almost surely, as long as the Markov category is \as{}-compatibly representable and has conditionals.
	To see this, consider two Bayesian inverses of $f$ with respect to $m$, $f_1^\dagger$ and $f_2^\dagger$, which thus have to satisfy
	\begin{equation}
		\input{fm_bayesian_inverse.tikz}
	\end{equation}
	by the definition of Bayesian inverses.
	Applying the sampling cancellation property and the causality property~\cite[Definition~11.30]{fritz2019synthetic} which follows from the existence of conditionals~\cite[Proposition~11.33]{fritz2019synthetic} gives
	\begin{equation}
		\input{fm_bayesian_inverse2.tikz}
	\end{equation}
	which is exactly the equation needed to conclude that $\hat{f}$ is well-defined up to $\as{m}$ equality.
	\begin{definition}\label{def:stnd_meas}
		The \emph{standard measure} $\hat{f}_m \colon I \to P \Theta$ of $f$ is then defined by
		\begin{equation}\label{eq:stnd_meas}
			\input{stndmeas.tikz}
		\end{equation}
	\end{definition}

	Per the $m$-almost sure uniqueness of $\hat{f}$, the standard measure $\hat{f}_m$ is unique as soon as $\cC$ is \as{}-compatibly representable.

	Having introduced the basic necessary ingredients of the theory of comparison of statistical experiments and the BSS Theorem in the language of Markov categories, we now present the results that build up to the BSS Theorem itself, assuming throughout that we are in an \as{}-compatibly representable Markov category with conditionals.
	
	\begin{lemma}\label{lem:stnd_enc}
		Let $f \colon \Theta \to X$ be a morphism with standard experiment $\hat{f} \colon \Theta \to P \Theta$ as defined above.
		Then the sampling map $\samp_{\Theta} \colon P \Theta \to \Theta$ is a Bayesian inverse of $\hat{f}$ with respect to $m$, i.e.~we have
		\begin{equation}
			\input{blackwelllemma1.tikz}
		\end{equation}
	\end{lemma}
	\begin{proof}
		Since, by definition, $(f^\dag)^{\sharp}$ is deterministic and satisfies \mbox{$\samp \circ (f^\dag)^{\sharp} = f^\dag$}, we can prove the lemma as follows:
		\begin{equation}
			\input{blackwelllemma2.tikz}
		\end{equation}
		
		\vspace*{-\baselineskip}
	\end{proof}
	
	The terminology ``standard experiment'' of $\hat{f}$ is then justified by the following result, which states that $\hat{f}$ is exactly as informative about $\Theta$ as the original experiment $f$ is, at least up to \as{$m$} equality.
	\begin{proposition}\label{prop:standard_experiment}
		For any $f$, we have $f \succeq \hat{f}$ and $\hat{f} \asinf{m} f$.
	\end{proposition}
	In standard measure-theoretic probability, this is \cite[Proposition~7.2.2]{torgersen1991comparison}.
	\begin{proof}
		Since $f \succeq \hat{f}$ is clear from the definition of the standard experiment, the crux of the proof lies in showing that there exists a morphism $r \colon P\Theta \to X$ such that $r \, \hat{f} \ase{m} f$.
		That is, $r$ is a garbling map which recovers $f$ from its standard version.
		We now show that choosing $r$ to be a Bayesian inverse of $(f^\dag)^{\sharp}$ with respect to $f \, m$ does the job.
		With this choice, we thus have
		\begin{equation}\label{eq:recovery_def}
			\input{recoverydef.tikz}
		\end{equation}
		by the definition of Bayesian inverses.
		
		Applying $\samp_\Theta \tensor \id_X$ to this equation yields
		\begin{equation}\label{eq:recovery_lhs}
			\input{recoverylhs.tikz}
		\end{equation}
		for its left-hand side and
		\begin{equation}\label{eq:recovery_rhs}
			\input{recoveryrhs.tikz}
		\end{equation}
		for its right hand side, where we use \Cref{lem:stnd_enc} to obtain the latter.
		Consequently, $f$ is $m$-almost surely equal to $r \, \hat{f}$, as we wanted to show.
	\end{proof}
	
	Now we have all the necessary ingredients to present a proof of our version of the Blackwell---Sherman---Stein Theorem in representable Markov categories.
	\begin{theorem}[Blackwell---Sherman---Stein]\label{thm:dilation_criterion}
		Let $\cC$ be an \as{}-compatibly representable Markov category with conditionals.
		Consider two morphisms $f \colon \Theta \to X$ and $g \colon \Theta \to Y$ in $\cC$, whose standard experiments are denoted by $\hat{f}$ and $\hat{g}$ respectively.
		
		Then the following are equivalent:
		\begin{enumerate}
			\item $f \asinf{m} g$, i.e.\ there exists a morphism $c \colon X \to Y$ such that
				\begin{equation}\label{eq:data_processing}
					c \, f \ase{m} g.
				\end{equation}
			\item $\hat{f}_m \sqsubseteq \hat{g}_m$, i.e.\ there exists a $\hat{g}_m$-dilation $t \colon P\Theta \to P\Theta$ such that 
			\begin{equation}\label{eq:posterior_dilation}
				\hat{f}_m = t \, \hat{g}_m.
			\end{equation}
		\end{enumerate}
	\end{theorem}
	\begin{proof}
		For the forward implication, let $c \colon P\Theta \to P\Theta$ be a morphism satisfying $c \, \hat{f} \ase{m} \hat{g}$, the existence of which is equivalent to $f \asinf{m} g$ by \Cref{prop:standard_experiment}.
		We then prove that $\hat{f}_m \sqsubseteq \hat{g}_m$ holds as well by constructing a $\hat{g}_m$-dilation that witnesses this second-order dominance relation. 
		Let $c^{\dagger}$ denote the Bayesian inverse of $c$ with respect to $\hat{f}_m$.
		Then $c^{\dagger}$ satisfies \cref{eq:posterior_dilation} by definition.
		In order to show that $c^\dag$ is a $\hat{g}_m$-dilation, we can use \Cref{lem:stnd_enc} twice:
		\begin{equation}\label{eq:blackwellthm1}
			\input{blackwellthm1_v2.tikz}
		\end{equation}
		Overall, $c^{\dagger}$ is thus a $\hat{g}_m$-dilation that maps $\hat{g}_m$ to $\hat{f}_m$, as was to be shown.
		
		For the converse implication, suppose that a $\hat{g}_m$-dilation $t$ with $\hat{f}_m = t \, \hat{g}_m$ exists, and denote its Bayesian inverse with respect to $\hat{g}_m$ by $t^{\dagger}$.
		We can now use the same steps as in the computation above in order to show that $t^{\dagger}$ achieves the conversion of $\hat{f}$ into $\hat{g}$, at least $m$-almost surely:
		\begin{equation}\label{eq:blackwellthm2}
			\input{blackwellthm2.tikz}
		\end{equation}
		We have thus shown that $t^\dagger$ achieves the conversion $\hat{f} \asinf{m} \hat{g}$ and consequently we get $f \asinf{m} g$ by \Cref{prop:standard_experiment}. 
	\end{proof}
	
	Traditionally, the BSS Theorem is stated with exact equalities rather then almost surely with respect some measure $m \in \cC_{\det}(I,P\Theta)$.
	This is because the theorem is usually applied to finite (or countably infinite) parameter sets $\Theta$, for which there exists a distribution $m \colon I \to \Theta$ having full support, so that $m$-almost sure equality coincides with plain equality,
	\begin{equation}\label{eq:full_support}
		f \ase{m} g  \quad\Longleftrightarrow\quad  f = g
	\end{equation}
	 for all $f,g \colon \Theta \to A$ to any other object $A$.
	 Thanks to this property, one can then remove all \as{$m$} qualifications. This generalizes as follows.
	\begin{definition}\label{def:discrete}
		An object $\Theta$ of a Markov category is termed \emph{discrete} if there exists a morphism $m \colon I \to \Theta$ such that \eqref{eq:full_support} holds for all $f,g \colon \Theta \to A$ to any $A$.
	\end{definition}
	Clearly such an $m$ satisfies $m \gg \mu$ for every other $\mu \colon I \to \Theta$ in the sense of \Cref{def:domination}.
	In $\BorelStoch$, the only standard Borel spaces $\Theta$ for which such an $m$ exists are the discrete measurable spaces, meaning that $\Theta$ must be finite or countably infinite.
	Then one can choose $m$ to be any distribution of full support on $\Theta$, where in the case of finite $\Theta$ the uniform distribution is the commonly used choice. 
	
	For discrete objects, we get the following version of the BSS Theorem, which is closer to the traditional account than \Cref{thm:dilation_criterion}.
	\begin{corollary}\label{thm:dilation_criterion_discrete}
		Let $\cC$ be an \as{}-compatibly representable Markov category with conditionals, and let $\Theta$ be a discrete object of $\cC$ with respect to $m \colon I \to \Theta$.
		Consider two morphisms $f \colon \Theta \to X$ and $g \colon \Theta \to Y$ in $\cC$, whose standard experiments are denoted by $\hat{f}$ and $\hat{g}$ respectively.
		
		Then the following are equivalent:
		\begin{enumerate}
			\item $f\succeq g$, i.e.\ there exists a morphism $c \colon X \to Y$ such that
				\begin{equation}
					c \, f = g.
				\end{equation}
			\item $\hat{f}_m \sqsubseteq \hat{g}_m$, i.e.\ there exists a $\hat{g}_m$-dilation $t \colon P\Theta \to P\Theta$ such that 
			\begin{equation}
				\hat{f}_m = t \, \hat{g}_m.
			\end{equation}
		\end{enumerate}
	\end{corollary}
	
	As far as we know, there is no established genuine generalization of the BSS Theorem beyond discrete parameter sets $\Theta$. 
	In \Cref{thm:dilation_criterion}, we have given a version of the theorem that goes beyond discrete parameters by virtue of using \as{} equality with respect to a measure $m$.
	In the specific case of $\cat{BorelStoch}$, the possibility of doing this has been known, see for example \cite{le1996comparison}.
	However, we believe that our synthetic treatment exposes the key reasons for this to be the case and can therefore catalyze further development.
	
	Arguably, the more interesting aspect of our \Cref{thm:dilation_criterion} is the fact that it applies in a much wider context than the traditional measure-theoretic one.
	In particular, we have shown that the result can be used in any Markov category that
	\begin{itemize}
		\item allows one to perform Bayesian inference---i.e.\ one that has conditionals, and
		\item can describe spaces of measures internally---i.e.\ one that is \as{}-compatibly representable.
	\end{itemize}
	So far, we have not even begun to explore the full scope of this type of result. 
	We explore one particular application in the next subsection.

\subsection{The Blackwell--Sherman--Stein Theorem Parametrized by Priors}\label{sec:blackwell3}

	In this subsection, we show that there is a (synthetic) version of the BSS Theorem which
	\begin{itemize}
		\item holds for an arbitrary hypothesis object $\Theta$, and
		\item does not refer to any particular prior $m$, 
	\end{itemize}
	in contrast to \Cref{thm:dilation_criterion} and \Cref{thm:dilation_criterion_discrete}, which only satisfy one of these requirements each. 
	The catch is that, in general, it characterizes a slightly different informativeness ordering, and can be thought of as a version of \Cref{thm:dilation_criterion} which is uniform in the prior.
	Interestingly, this result actually arises as a \emph{special case} of \Cref{thm:dilation_criterion}, namely when the latter is instantiated in a parametric Markov category as introduced in \Cref{sec:parametric}.

	As before, we think of $f$ and $g$ as statistical experiments with hypothesis space or parameter space $\Theta$.
	Given the corresponding distribution object $P\Theta$ associated to the hypothesis space $\Theta$, we then consider the parametric Markov category $\cC_{P \Theta}$.
	Intuitively, we think of morphisms in $\cC_{P \Theta}$ as Markov kernels parametrized by a prior over the hypothesis space $\Theta$. 
	Throughout this section, we use \as{} equality in $\cC_{P \Theta}$ with respect to the global element $\param{\prsamp \in \cC_{P \Theta}(I,\Theta)}$ represented by the sampling map:
	\begin{equation}\label{eq:param_full_support}
		\input{param_full_support.tikz}
	\end{equation}
	We think of this morphism as sampling a hypothesis in $\Theta$ distributed in accordance with the prior (which is not fixed here, but rather an extra parameter).
	
	Before we discuss the BSS Theorem instantiated in $\cC_{P \Theta}$, we consider how the basic notions of comparison of statistical experiments look in $\cC_{P \Theta}$ when we reexpress them in terms of the 
		corresponding morphisms in $\cC$.
	
	First of all, the statistical models $\Theta \to X$ are thought of as models of the behavior of a system as a function of a model parameter $\Theta$, typically not under the experimenter's control.
	The point of the theory of statistical experiments is to formalize and quantify the procedure of making inferences about an \emph{unknown} parameter $\Theta$ by virtue of learning the value of $X$.
	We thus still assume that the statistical models in $\cC_{P \Theta}$ are morphisms independent of the prior $P \Theta$ and are therefore represented by
	\begin{equation}\label{eq:param_model}
		\input{param_model.tikz}
	\end{equation}
	For better readability in future expressions in the $\cC$-representation, we introduce a new notation for the \emph{prior behavior} $\mathfrak{f} \colon P \Theta \to \Theta \otimes X$ of a statistical model $\param{f}$ as follows
	\begin{equation}\label{eq:param_parallel}
		\input{param_parallel.tikz}
	\end{equation}
	The morphism $\mathfrak{f}$ describes what a Bayesian experimenter expects to observe: 
	For each prior distribution over hypotheses, it returns as outputs an experiment outcome, for a hypothesis randomly sampled from the prior, together with that hypothesis.

	The process of Bayesian updating itself is then described by
	\begin{equation*}
		\bigl(\mathfrak{f}_{|X}\bigr)^\sharp \colon X \otimes P\Theta \longrightarrow P\Theta,
	\end{equation*}
	where $\mathfrak{f}_{|X}$ denotes a conditional of $\mathfrak{f}$ with respect to $X$.
	The deterministic morphism $(\mathfrak{f}_{|X})^\sharp$ takes an outcome and a prior as input and returns the associated Bayesian posterior. 
	Since conditionals are not unique in general, this Bayesian updating map is not uniquely determined by $f$ itself.
	As before, Bayesian updating features in the upcoming definition of the standard experiment.
	
	But let us consider the comparison of statistical experiments in $\cC_{P\Theta}$ first.
	Interpreting \Cref{def:informativeness} in $\cC_{P\Theta}$, the $\param{\as{\prsamp}}$ informativeness ordering of two prior-independent morphisms as in \cref{eq:param_model} says that $\param{f \asinf{\prsamp} g}$ if and only if there exists $c \colon X \tensor P\Theta \to Y$ such that we have
	\begin{equation}\label{eq:parametrized_informativeness}
		\input{parametrized_informativeness_v2.tikz}
	\end{equation}
	This amounts to allowing the garbling map $c$ in the definition of the informativeness ordering to depend on the prior in addition to its dependence on the data variable of the experiment $f$ used to simulate $g$. 
	Since the condition (\ref{eq:parametrized_informativeness}) gives rise to an ordering on morphisms of $\cC$ that differs from those considered in \Cref{sec:informativeness,sec:BSS}, we give it a new name.
	\begin{notation}
		Given two morphisms $f$ and $g$ in a Markov category $\cC$ with a common domain $\Theta$, we say that $f$ is \emph{more informative} than $g$ \emph{in the Bayesian sense}, denoted $f \Bayesinf g$, if we have $\param{f \asinf{\prsamp} g}$ in $\cC_{P \Theta}$ as expressed by the existence of a prior-dependent garbling $c$ that achieves the conversion shown in \cref{eq:parametrized_informativeness}.
	\end{notation}
	
	Before the exposition of our parametric BSS Theorem itself, we present some results on how the Bayesian informativeness ordering relates to the usual prior-independent one. 
	While it is obvious that the existence of a prior-independent $c$ with $c \, f = g$ implies the existence of a prior-dependent $c$ as above---namely by choosing that dependence to be the trivial one which discards the prior---the question of whether the converse also holds is far more involved. 
	We start addressing it with the following simple observation, which shows that the garbling map can be chosen in a uniform way for all distributions which are absolutely continuous with respect to a given one.
	
	\begin{proposition}\label{prop:informativeness_as}
		Let $\nu \colon I \to \Theta$ be arbitrary.
		If $f$ is more informative than $g$ in the Bayesian sense, then there exists a morphism $c_{\nu} \colon X \to Y$ in $\cC$ such that we have $c_\nu \, f \ase{\mu} g$, i.e.\
		\begin{equation}\label{eq:informativeness_n-as}
			\input{informativeness_n-as_v2.tikz}
		\end{equation}
		for every $\mu \colon I \to \Theta$ with $\nu \gg \mu$.
	\end{proposition}
	Recall that $\nu \gg \mu$ denotes the absolute continuity ordering from \Cref{def:domination}.
	\begin{proof}
		In particular, $c_{\nu}$ can be constructed from $c$ in \cref{eq:parametrized_informativeness} as
		\begin{equation}
			\label{eq:informativeness_n-as2}
			\input{conditionalconversion_v2.tikz}
		\end{equation}
		which makes \eqref{eq:informativeness_n-as} hold with $\nu$ in place of $\mu$. The claim then follows from the definition of $\nu \gg \mu$.
	\end{proof}
	
	\begin{corollary}\label{cor:discrete_reduction}
		If the hypothesis object $\Theta$ is discrete, then
		\[
			f \succeq g \qquad \Longleftrightarrow \qquad f \Bayesinf g.
		\]
	\end{corollary}

	\begin{proof}
		Discreteness implies that there is a distribution $\nu \colon I \to \Theta$ such that taking $\mu \coloneqq \nu$ makes $f \succeq g$ follow from \cref{eq:informativeness_n-as}.
	\end{proof}
	
	However, the existence of a prior-dependent garbling map does not imply the existence of a completely prior-independent one in general.
	We now present an explicit counterexample in $\BorelStoch$ based on an example due to Blackwell and Ramamoorthi~\cite{BR}.

	\begin{proposition}\label{prop:garbling_counterexample}
		In $\BorelStoch$, there are $f \colon \Theta \to X$ and $g \colon \Theta \to Y$ such that a garbling $c$ satisfying \cref{eq:parametrized_informativeness} exists, but there is no Markov kernel $c \colon X \to Y$ with $g = c \, f$.
	\end{proposition}
	\begin{proof}
		The following arguments amount to showing that the example of Blackwell and Ramamoorthi~\cite{BR}, which proved that classical sufficiency and Bayesian sufficiency are not equivalent, also similarly in our context of comparison of experiments. The main difference is that our version of their example requires a more refined measurability analysis in the second part.

		As the spaces of outcomes of the experiments, consider the two standard Borel spaces
		\begin{align*}
			X &= \{a,b\}^\N,  &  Y &= \{a,b\}.
		\end{align*}
		Writing $\pi_n \colon \{a,b\}^\N \to \{a,b\}$ for the $n$-th product projection, we define the sets
		\begin{align*}
			\Theta_a	& \coloneqq \Set{ \nu \in PX \given \lim_{n \to \infty} \nu\bigl(\pi_n^{-1}(a)\bigr) = 1 },	\\[2pt]
			\Theta_b	& \coloneqq \Set{ \nu \in PX \given \lim_{n \to \infty} \nu\bigl(\pi_n^{-1}(b)\bigr) = 1 },
		\end{align*}
		to be thought of as containing those distributions for which the countably many $\{a,b\}$-valued random variables represented by a distribution on $X$ converge in probability to $a$ or $b$,
		respectively. 
		Since we have
		\[
			\Theta_a = \bigcap_{k \in \N} \, \bigcup_{n \in \N} \, \Set*[\Big]{ \nu \in PX \given \nu\bigl(\pi_n^{-1}(b)\bigr) \le 2^{-k} },
		\]
		the set $\Theta_a$ is a measurable subset of $PX$, and we take it to be equipped with the induced $\sigma$-algebra.
		Moreover, it constitutes a standard Borel space since $PX$ does too.
		A similar argument shows that $\Theta_b$ is likewise a standard Borel space. 

		We now consider the (disjoint) union $\Theta \coloneqq \Theta_a \cup \Theta_b$, and let a measurable map $f \colon \Theta \to X$ be defined by the restriction of $\samp \colon PX \to X$ to $\Theta \subseteq PX$, while the measurable map $g \colon \Theta \to Y$ is given by
		\[
			g(\nu) \coloneqq	
					\begin{cases}
						a	& \text{if } \nu \in \Theta_a,	\\
						b	& \text{if } \nu \in \Theta_b.
					\end{cases}
		\]
		That is, $g(\nu) = a$ indicates that the sequence of random variables distributed according to $\nu$ converges in probability to $a$, and similarly for $g(\nu) = b$.

		We first argue that there is no garbling from $f$ to $g$, meaning that there is no Markov kernel $c \colon X \to Y$ with $g = c \, f$. 
		If such a $c$ did exist, then one could define a measurable map $s \colon X \to \{a,b\}$ by
		\[
			s(x) \coloneqq	
				\begin{cases}
					a	& \textrm{if } c \bigl( \{a\} \big| x \bigr) > 0,	\\
					b	& \textrm{if } c \bigl( \{a\} \big| x \bigr) = 0.
				\end{cases}
		\]
		Since for $\nu \in \Theta_a$ we have $c(\{a\}|\ph) \ase{\nu} 1$, it follows that also $s \ase{\nu} a$ holds. 
		Similarly, for every $\nu \in \Theta_b$ we have $s \ase{\nu} b$. 
		However, Blackwell has shown in \cite{blackwell_splifs} that such an $s$ does not exist, and therefore neither does $c$.

		Second, we construct a prior-dependent $c \colon X \otimes P\Theta \to Y$ such that \cref{eq:parametrized_informativeness} holds. 
		In particular, note that even though $\Theta$ is defined so that the sequence of random variables converges in probability to either $a$ or $b$, we cannot extract the limit from a sample sequence alone per the argument of the previous paragraph. 
		We thus show that, given a prior $\mu \in P\Theta$, one can always find a subsequence of the random variables that converges $\nu$-\emph{almost surely} for $\mu$-almost all measures $\nu \in \Theta$.
		
		This requires a bit of preparation. 
		For fixed $\mu \in P\Theta$, consider the average measures defined for all $S \in \Sigma_X$ as
		\begin{align}
			\overline{\nu}_a(S) &\coloneqq \int_{\Theta_a} \! \nu(S) \, \mu(d\nu), &
			\overline{\nu}_b(S) &\coloneqq \int_{\Theta_b} \! \nu(S) \, \mu(d\nu).
		\end{align}
		Note that these are both subnormalized, and that $\overline{\nu}_a + \overline{\nu}_b$ is exactly the expected probability measure $(Pf)(\mu)$ on $X$. 
		By the monotone convergence theorem, the assumed convergence in probability of the product projections still holds on average, i.e.\ we have
		\begin{align}\label{eq:convergence_in_prob1}
			\lim_{n \to \infty} \overline{\nu}_a \bigl( \pi_n^{-1}(b) \bigr) &= 0,  &
			\lim_{n \to \infty} \overline{\nu}_b \bigl(\pi_n^{-1}(a) \bigr) &= 0.
		\end{align}
		For $k \in \N$, let then $n_k(\mu)$ be the smallest natural number such that both the inequalities
		\begin{align}
			\label{eq:convergence_in_prob}
			\overline{\nu}_a \bigl(\pi_{n_k(\mu)}^{-1}(b) \bigr) &\le 2^{-k} & 
			\overline{\nu}_b \bigl(\pi_{n_k(\mu)}^{-1}(a) \bigr) &\le 2^{-k}
		\end{align}
		are satisfied.
		The existence of $n_k(\mu)$ is guaranteed by (\ref{eq:convergence_in_prob1}).
		
		We now show that the map $\mu \mapsto n_k(\mu)$ is measurable. To see this, notice that for every $m \in \N$ we have $n_k(\mu) \le m$ if and only if both
		\[
			\int 1_{\Theta_a} \, \nu \bigl( \pi_m^{-1}(b) \bigr) \, \mu(d\nu) \le 2^{-k} \quad \text{ and } \quad
			\int 1_{\Theta_b} \, \nu \bigl( \pi_m^{-1}(a) \bigr) \, \mu(d\nu) \le 2^{-k}
		\]
		are satisfied.
		Preimages of the upper sets $\Set{m \given k \le m} \subseteq \mathbb{N}$ are thus measurable by the inequalities above, because integration of a fixed measurable function on $\Theta$ is measurable in $\mu \in P\Theta$.
		Since the discrete $\sigma$-algebra on $\N$ is generated by the upper sets, the assignment $\mu \mapsto n_k(\mu)$ is measurable.

		Thanks to inequalities~\eqref{eq:convergence_in_prob} and the Borel-Cantelli lemma, we have that the measurable sets
		\begin{align*}
			S_a(\mu)	& \coloneqq \Set{ x \in X \given \pi_{n_k(\mu)}(x) = a \textrm{ for infinitely many } k },	\\
			S_b(\mu)	& \coloneqq \Set{ x \in X \given \pi_{n_k(\mu)}(x) = b \textrm{ for infinitely many } k },
		\end{align*}
		satisfy
		\[
			\overline{\nu}_a \bigl( S_b(\mu) \bigr) = \overline{\nu}_b \bigl( S_a(\mu) \bigr) = 0.
		\]
		But since this must then also hold $\mu$-almost surely for all the measures which form the averages $\overline{\nu}_a$ and $\overline{\nu}_b$, it follows that $\nu(S_b) = 0$ also holds for $\mu$-almost all $\nu \in \Theta_a$, and similarly $\nu(S_a) = 0$ for $\mu$-almost all $\nu \in \Theta_b$. 
		But then we can decide whether $\nu \in \Theta_a$ or $\nu \in \Theta_b$ simply by testing membership of its sample in $S_a$.
		In other words, the function
		\[
			c \colon X \otimes P\Theta \rightarrow \{a,b\}, \qquad
			(x,\mu) \mapsto	
				\begin{cases}
					a	& \text{if } x \in S_a(\mu),	\\
					b	& \text{if } x \in \overline{S_a(\mu)},
				\end{cases}
		\]
		makes the desired \cref{eq:dataprocessing2} hold, since $c(x,\mu) = a$ for $\nu$-almost every $x \in X$ and $\mu$-almost every $\nu \in \Theta_a$, and similarly $c(x,\mu) = b$ for $\nu$-almost every $x$ and $\mu$-almost every $\nu \in \Theta_b$.

		It remains to be shown that $c$ is actually measurable. 
		This follows because the complement $\overline{S_a(\mu)}$ is given by the countable disjoint union
		\[
			\overline{S_a(\mu)} \,=\, \bigcup_{F} \, \Set*[\Big]{ x \in \{a,b\}^\N \given \pi_{n_k(\mu)}(x) = b \iff k \in F }
		\]
		over finite $F \subseteq \mathbb{N}$.
		Since for every fixed $F$ the set of all pairs $(x,\mu)$ for which the internal condition holds is measurable by the measurability of $\mu \mapsto n_k(\mu)$, the set of all pairs $(x,\mu)$ with $c(x,\mu) = b$ is likewise measurable. 
		Therefore, $c$ is measurable, thereby making $f$ indeed more informative than $g$ in the Bayesian sense.
	\end{proof}
	
	We continue with the general theory, aiming at a BSS Theorem for characterizing informativeness in the Bayesian sense.

	The standard experiment of a statistical model $\param{f \colon \Theta \to X}$ in $\cC_{P\Theta}$, as introduced in \Cref{def:stn_exp}, now takes the form
	\begin{equation}\label{eq:stndexp3}
		\input{stndexp3.tikz}
	\end{equation}
	Similar to before, this (typically non-deterministic) morphism takes a hypothesis as its first argument, a prior as its second argument, and outputs the distribution over posteriors which results if the given hypothesis is true and Bayesian updating is used with respect to the given prior. 
	Again as before, in order for $\param{\hat{f}}$ to be well defined up to $\param{\as{\prsamp}}$ equality, we need $\cC_{P\Theta}$ to be \as{}-compatibly representable as opposed to merely being representable.
	However, this is guaranteed by \Cref{ass:ascr} and \Cref{lem:parametric_ascr}.
	
	Next, the standard measure $\param{\hat{f}_{\prsamp} \coloneqq \hat{f} \circ \prsamp}$ in $\cC_{P\Theta}$ is represented in $\cC$ by
	\begin{equation}\label{eq:stndmeas2}
		\input{stndmeas2.tikz}
	\end{equation}
	and \Cref{lem:stnd_enc} becomes:
	\begin{equation}\label{eq:stnd_enc3}
		\input{stnd_enc3.tikz}
	\end{equation}
	Consequently, the statement of \Cref{thm:dilation_criterion} in $\cC_{P\Theta}$ with respect to $\param{\prsamp}$ reads as follows.
	\begin{corollary}\label{thm:dilation_criterion2}
		Let $\cC$ be an \as{}-compatibly representable Markov category with conditionals.
		Consider two morphisms $f \colon \Theta \to X$ and $g \colon \Theta \to Y$ in $\cC$ with $\hat{f}_\prsamp$ and $\hat{g}_\prsamp$ given by the right hand side of \cref{eq:stndmeas2}.
		
		Then the following are equivalent:
		\begin{enumerate}
			\item \label{it:informativeness_criterion2} $f \Bayesinf g$, i.e.\ there exists a morphism $c \colon X \tensor P\Theta \to Y$ satisfying
				\begin{equation}\label{eq:dataprocessing2}
					\input{dataprocessing2_blue.tikz}
				\end{equation}
				
			\item \label{it:dilation_criterion2} $\param{\hat{f}_{\prsamp} \sqsubseteq \hat{g}_{\prsamp}}$, i.e.\ there exists a $\hat{g}_\prsamp$-dilation $t \colon P\Theta \tensor P\Theta \to P\Theta$ satisfying 
			\begin{equation}\label{eq:stndmeaspreserving2}
				\input{stndmeaspreserving2.tikz}
			\end{equation}
		\end{enumerate}
	\end{corollary}
	Indeed, one can easily check that every $\param{\hat{g}_{\prsamp}}$-dilation in $\cC_{P\Theta}$ is represented by a $\hat{g}_\prsamp$-dilation in $\cC$ (and vice versa), so that the claim follows.
	
	Instantiating this in $\BorelStoch$ results in the following: 
	A statistical experiment represented by a Markov kernel $f$ is more informative than one represented by $g$ for \emph{every} prior if and only if the resulting expected distribution over posteriors of the first experiment is more spread out than that of the second experiment for \emph{every} prior, where in both parts of this statement the dependence on the prior is assumed measurable.

\bibliographystyle{plainurl}
\bibliography{markov}

\end{document}